\documentclass{article}
\usepackage{amsmath,amsthm,amsfonts,graphicx}
\usepackage{multirow}
\usepackage{slashbox}
\usepackage{multirow}
\def\smallddots{\mathinner{\raise7pt\hbox{.}\raise4pt\hbox{.}\raise1pt\hbox{.}}}
\def\smallsdots{\mathinner{\raise1pt\hbox{.}\raise4pt\hbox{.}\raise7pt\hbox{.}}}
\newtheorem{theorem}{Theorem}
\newtheorem{lemma}[theorem]{Lemma}
\newtheorem{corollary}[theorem]{Corollary}
\theoremstyle{definition}
\newtheorem{definition}[theorem]{Definition}

\newtheorem{remark}[theorem]{Remark}

\DeclareMathOperator{\diag}{diag}

\DeclareMathOperator{\rank}{rank}
\begin{document}
\title{Transformations of Matrix Structures Work Again}
\author{Victor Y. Pan$^{[1, 2],[a]}$ 
\and\\
$^{[1]}$ Department of Mathematics and Computer Science \\
Lehman College of the City University of New York \\
Bronx, NY 10468 USA \\
$^{[2]}$ Ph.D. Programs in Mathematics  and Computer Science \\
The Graduate Center of the City University of New York \\
New York, NY 10036 USA \\
$^{[a]}$ victor.pan@lehman.cuny.edu \\
http://comet.lehman.cuny.edu/vpan/  \\
} 
\date{}
\maketitle

\begin{abstract}
In \cite{P90} we proposed to employ Vandermonde and Hankel multipliers to transform 
into each other the matrix structures of  Toep\-litz, Han\-kel, Van\-der\-monde 
and Cauchy types as a means of extending any successful algorithm for the 
inversion of matrices having one of these structures to inverting the matrices 
with the structures of the three other types. Surprising power of this approach 
has been demonstrated in a number of works, which culminated in ingeneous
numerically stable algorithms that approximated the solution of a nonsingular 
Toeplitz linear system in nearly linear (versus previuosly cubic) arithmetic time. 
We first revisit this powerful method, covering it comprehensively, and then 
specialize it to yield a similar acceleration of the known algorithms for 
computations with matrices having structures of Van\-der\-monde or  Cauchy types. 
In particular we arrive at numerically stable approximate multipoint polynomial 
evaluation and interpolation in nearly linear time, by using $O(bn\log^hn)$ flops 
where $h=1$ for evaluation, $h=2$ for interpolation, and $2^{-b}$ is the relative 
norm of the approximation errors.
 
\end{abstract}

\paragraph{Keywords:} 
Transforms of matrix structures,  
Van\-der\-monde matrices,
Cauchy matrices,
Multipole method,
HSS matrices,
Polynomials,
Rational functions,
Multipoint evaluation,
Interpolation

\paragraph{AMS Subject Classification:}
15A04, 15A06, 15A09, 47A65, 65D05, 65F05, 68Q25

\section{Introduction}

\begin{table}[ht]
\caption{Four classes of structured matrices}
\label{t1}.
\begin{center}
\renewcommand{\arraystretch}{1.2}
\begin{tabular}{c|c}
Toeplitz matrices $T=\left(t_{i-j}\right)_{i,j=1}^{n}$&Hankel matrices $H=\left(h_{i+j}\right)_{i,j=0}^{n-1}$ \\
$\begin{pmatrix}t_0&t_{-1}&\cdots&t_{1-n}\\ t_1&t_0&\smallddots&\vdots\\ \vdots&\smallddots&\smallddots&t_{-1}\\ t_{n-1}&\cdots&t_1&t_0\end{pmatrix}$&$\begin{pmatrix}h_0&h_1&\cdots&h_{n-1}\\ h_1&h_2&\smallsdots&h_n\\ \vdots&\smallsdots&\smallsdots&\vdots\\ h_{n-1}&h_n&\cdots&h_{2n-2}\end{pmatrix}$ 
\\ \hline 


Van\-der\-monde matrices 
$V=V_{\bf s}=\left(s_i^{j}\right)_{i,j=0}^{n-1}$&Cauchy matrices $C=C_{\bf s,t}=\left(\frac{1}{s_i-t_j}\right)_{i,j=1}^{n}$ \\
$\begin{pmatrix}1&s_1&\cdots&s_1^{n-1}\\ 1&s_2&\cdots&s_2^{n-1}\\ \vdots&\vdots&&\vdots\\ 1&s_{n}&\cdots&s_{n}^{n-1}\end{pmatrix}$&$\begin{pmatrix}\frac{1}{s_1-t_1}&\cdots&\frac{1}{s_1-t_{n}}\\ \frac{1}{s_2-t_1}&\cdots&\frac{1}{s_2-t_{n}}\\ \vdots&&\vdots\\ \frac{1}{s_{n}-t_1}&\cdots&\frac{1}{s_{n}-t_{n}}\end{pmatrix}$
\end{tabular}
\end{center}
\end{table}

Table \ref{t1} displays four classes of most popular structured matrices,
which
are omnipresent in modern computations for Sciences, Engineering, and Signal
and Image Processing and
which have been naturally extended to larger classes of matrices, 
$\mathcal T$, $\mathcal H$, $\mathcal V$, and $\mathcal C$,
having structures of Toep\-litz, Hankel,
Van\-der\-monde and Cauchy types,
 respectively. 
Such matrices 
can be readily expressed via
their displacements of small ranks, which implies
a number of their further attractive properties:

\begin{itemize}\itemsep=-0.6mm\vspace{-2mm}
\item
Compact compressed representation through
a small number of parameters, typically $O(n)$ parameters
in the case of $n\times n$  matrices
\item
Simple expressions for the inverse 
through the solutions of a small number
of linear systems of equations
wherever the matrix is invertible
\item
Multiplication by vectors in nearly linear arithmetic time
\item
Solution of nonsingular
linear systems of equations with these matrices
in quadratic or nearly linear arithmetic time
\vspace{-2mm}
\end{itemize}

Extensive and highly successful research and implementation work based
on these properties has been continuing for more than three decades.
We follow \cite{P90}
and employ structured matrix multiplications to
 transform the four structures
into each other. For example, 
$\mathcal T \mathcal H=\mathcal H\mathcal T=\mathcal H$, 
$\mathcal H\mathcal H=\mathcal T$, and $V^TV$ is a Hankel matrix.
The paper \cite{P90} showed 
that {\em this technique enables one
  to extend any successful algorithm for  
the
inversion 
of the matrices of any
  of the four classes 
 $\mathcal T$, $\mathcal H$, 
$\mathcal V$, and $\mathcal C$
to the  matrices of the  
three other classes}.
We cover this technique comprehensively
and simplify its presentation versus  \cite{P90} because 
instead of the the
Stein displacements $M-AMB$ in \cite{P90}
we  
employ the Sylvester displacements $AM-MB$ 
and the machinery of operating with them from \cite{P00} and 
 \cite[Section 1.5]{P01}. 

The proposed structure transforms are simple 
but have surprising power where
the transform links
  matrix classes
having distinct
features.
For example, 
the 
matrix
structure of Cauchy type
is invariant in 
row and column interchange
(in contrast to the structures of
Toeplitz and Hankel types)
and enables expansion of
the matrix entries into Loran's series
(unlike the structures of
the three other types).
 Exploiting 
these distinctions
has lead to 
dramatic acceleration of 
the known numerically stable algorithms for 
Toeplitz and Toep\-litz-like linear systems of equations
by  means of their transformation  into Cauchy-like matrices 
and exploiting the above properties of these matrices.

Their invariance to row 
interchange
enabled numerically stable solution in 
quadratic rather than cubic time in
 \cite{GKO95}, \cite{G98}, \cite{R06}, 
but the paper \cite{MRT05}  
(cf. also \cite{CGS07}, \cite{XXG12},   and  \cite{XXCB})
has instead exploited the Loran's expansion of 
the entries of the basic Cauchy matrices 
to obtain their close approximation by HSS 
matrices. (``HSS" is the acronym for ``hierarchically semiseparable".)
This structure of a distinct type
enabled application
of the {\em Multipole/HSS} powerful techniques,
and the resulting 
numerically stable algorithms
approximate the solution 
of a nonsingular  Toeplitz linear system
of equations
in nearly linear (and thus
nearly optimal) arithmetic
time. The intensive work 
in  \cite{XXG12} and  \cite{XXCB}
on extension, refinement and implementation 
of the algorithms has already made
them quite attractive for the users.

Similar advance has not been achieved,
however, for the computations with
matrices having structures of Van\-der\-monde or Cauchy types.
All the cited papers
on Toeplitz computations share their basic 
displacement map, which 
is a specialization of our general
class of the transformations of matrix structures
derived from \cite{P90} (see our comments at the end of Section \ref{scv}).
The 
map
transforms the matrices with the structure
of Toeplitz type into the matrices of the subclasses
of the class $\mathcal C$
linked to FFT and defined by the knot sets 
$\{s_1,\dots,s_n,t_1,\dots,t_n\}$
equally spaced on the unit circle 
$\{z:~|z|=1\}$ of the complex plane.
This covers 
the structures of
Toeplitz but not 
Van\-der\-monde and Cauchy types.

In our present paper
we specify the subclass of {\em CV} and {\em CV-like} matrices,
which are the Cauchy and Cauchy-like matrices, respectively,
having at least
one (but not necessarily both)
of their two basic knot sets $\{s_1,\dots,s_n\}$ or $\{t_1,\dots,t_n\}$
 equally spaced on the unit circle $\{z:~|z|=1\}$.
These are precisely the Cauchy and Cauchy-like matrices that have
FFT-type structured transforms
into the matrices of the class $\mathcal V$
or their 
transposes.
Under this framework 
our main technical step is an extension of the algorithms of
 \cite{MRT05}, \cite{CGS07}, and \cite{XXG12} to proving that
all CV and CV-like matrices can be closely appoximated 
by HSS matrices. As soon as
such an approximation is available,
one just needs to
apply the Multipole
method to the HSS matrices to  obtain 
 numerically stable approximation 
 algorithms that run in nearly linear time
for our tasks for CV matrices, versus quadratic time of the known algorithms.
By applying the FFT-based structured transforms between 
matrices with the structures of CV and Vandermonde types, 
we readily extend these results to the matrices of the latter class and
consequently to the problems of multipoint evaluation and interpolation for 
polynomials. 

The new algorithms approximate within  relative error norm bound $2^{-b}$
 the product of an $n\times n$ CV matrix  
by a vector by using $O(bn\log n)$ flops
and the solution of a nonsingular CV linear system of $n$
equations by using $O(bn\log^2n)$ flops.
FFT-based structured transforms 
extend these algorithms and complexity bounds to 
computations with
Van\-der\-monde matrices 
and to
approximate
multipoint evaluation and interpolation for 
polynomials.
The resulting nearly linear time bounds are nearly optimal,
but still seem to be overly pessimistic, in view 
of the results
of the extensive tests in  \cite{XXG12} 
for the similar HSS computations (see our Remark \ref{reimpl}).
The cited results are readily extended to CV-like matrices
and consequently to the matrices having structure of Van\-der\-monde type.

Various  extensions,
ameliorations,  refinements,
and nontrivial  specializations
of the proposed  methods
can be interesting. Most valuable 
would be
new transforms among 
various new classes of structured matrices,
with significant algorithmic applications.
At the end of Section 
\ref{svmv}
we sketch a natural extension of our techniques to the general class of 
Cauchy and Cauchy-like matrices, but
indicate that this generally
 complicates the control over the output errors.
It can be interesting that even a very 
crude variant of our techniques 
(which proceeds with a limited use of the HSS algorithms)
still accelerates the known numerical algorithms
for multipoint polynomial evaluation by a factor of $\sqrt {n/\log n}$
(see Remark \ref{recrd}).

For a sample further application, recall that
the current best package of subroutines for polynomial
root-finding, MPSolve, is reduced essentially to
recursive  application of the   Ehr\-lich--Aberth algorithm,
and consequently to recursive numerical multipoint polynomial evaluation.
For this task  MPSolve uses a quadratic time algorithm,
which is the users'  current choice.
So our present acceleration from quadratic to a nearly linear time
can be translated into the same acceleration of MPSolve.

We organize our pesentation  as follows.
After recalling some definitions and basic facts
on general  matrices  and on four classes of structured matrices
$\mathcal T$, $\mathcal H$, $\mathcal V$, and $\mathcal C$
in the next three sections, we cover in some detail the 
transformations of matrix structures among these classes
in Section \ref{sdtr}, recall 
the class of HSS matrices in Section \ref{shss},
estimate numerical ranks of Cauchy and Cauchy-like
 matrices of a large class
in Section \ref{snrqs}, and extend these estimates to 
compute the HSS approximations of these matrices in 
Section \ref{shssapr} and to approximate the products
of these matrices and their inverses by a vector
in Section \ref{svmv}. We conclude the paper with Section \ref{scnc}.

For simplicity we assume square structured matrices throughout,
but our study can be readily extended to the case 
of rectangular matrices.

\section{Some definitions and basic facts}\label{sdef}

Hereafter ``flop" stands for ``arithmetic operation"; 
 the concepts ``large", ``small", ``near", ``close", ``approximate", 
``ill  conditioned" and ``well conditioned" 
are 
quantified in the context. 
Next we recall and extend some basic definitions  
and facts 
on computations with general and structured  matrices
(cf. \cite{GL96}, \cite{S98}, \cite{P01}).


\subsection{General matrices}\label{sgen}



$M=(m_{i,j})_{i,j=1}^{m,n}$ is an $m\times n$ matrix,
$M^T$ and $M^H$ are its transpose 
and  Hermitian (complex conjugate)
transpose,
respectively. 
We write $M^{-T}$ for $(M^T)^{-1}=(M^{-1})^T$.

$(B_1~|~\dots~|~B_n)$ denotes a $1\times n$ block matrix with the blocks $B_1,\dots,B_n$.
$\diag(B_1,\dots,B_n)=\diag(B_j)_{j=1}^n$ is an $n\times n$ block diagonal matrix with 
the diagonal blocks $B_1,\dots,B_n$. 
In the case of scalar blocks $s_1,\dots,s_n$
we arrive at a vector ${\bf s}=(s_j)_{j=1}^n$ and
an $n\times n$ diagonal matrix $D_{\bf s}=\diag({\bf s})=\diag(s_j)_{j=1}^n$ with 
the diagonal entries $s_1,\dots,s_n$.

The $n$ coordinate vectors
${\bf e}_1,\dots,{\bf e}_n$ of a dimension $n$ form
the $n\times n$ identity
matrix $I_n=({\bf e}_1~|~\dots~|~{\bf e}_n)$ and 
 the  $n\times n$ reflection matrix $J_n=({\bf e}_n~|~\dots~|~{\bf e}_1)$.
 $J_n=J_n^T=J_n^{-1}$.
We write $I$ and $J$
where the matrix size  
is not important or is defined by context.

{\bf Preprocessors.} For three nonsingular matrices $P$, $M$, and $N$
 and a vector ${\bf b}$,
the equations
\begin{equation}\label{eqprepr}
M^{-1}=N(PMN)^{-1}P,~~~ 
PMN{\bf y}=P{\bf b},~~{\bf x}=N{\bf y} 
\end{equation}
 reduce the
inversion of the matrix $M$ and the solution of 
a linear system
of equations $M{\bf x}={\bf b}$
to the inversion of the product $PMN$ 
and the solution of 
the linear system 
$PMN{\bf y}=P{\bf b}$,
respectively.
For some important classes of matrices $M$
this preprocessing can simplify 
dramatically the inversion of a matrix 
and the solution of a linear system of equations.

{\bf Generators.}
Given an $m\times n$ matrix $M$ of a rank $r$ 
and an integer $l\ge r$, we have a
nonunique expression
$M=FG^T$ for pairs $(F,G)$ of matrices of sizes 
$m\times l$ and $n\times l$, respectively. 
We call such a pair $(F,G)$
a {\em generator of length} $l$ for the matrix $M$, which
 is the shortest for $l=r$.

\begin{theorem}\label{thcompr} (Cf. \cite{BA80}, \cite{M80}, \cite{P93}, \cite{GE96}, \cite[Section 4.6.2]{P01}.)
Given a generator of a length $l$  for an $n\times n$ matrix $M$ having a rank $r$, $r\le l\le n$, 
it is sufficient to use $O(l^2n)$ flops to compute a generator of length $r$ for the matrix $M$.
\end{theorem}

{\bf Norm, conditioning, orthogonality, numerical rank, a perturbation norm bound.}
$||M||=||M^H||=||M||_2$ is the (Euclidean) 2-norm of a matrix $M$.


For a fixed tolerance $\tau$
the minimum rank of matrices 
in the $\tau$-neighborhood of a matrix $M$
is said to be its $\tau$-{\em rank}.
The  {\em numerical rank} of a matrix 
is its $\tau$-rank for a small positive $\tau$.
A matrix is called ill conditioned 
if it has a close neighbor of a smaller rank
or eqiuivalently
if its rank exceeds its numerical rank.
Otherwise it
is called {\em well conditioned}. If
a matrix $M$ is 
ill conditioned,
one must compute its inverse and 
the solution of a linear system $M{\bf x}={\bf f}$
with a high precision
to ensure meaningful output for these problems, but 
{\em  not for multiplication} 
by a vector.

A matrix $M$ is {\em unitary} or {\em orthogonal} if $M^HM=I$ or $MM^H=I$.
It is {\em quasiunitary} if $cM$ is unitary for a constant $c$.
Such a matrix $U$ has full rank and is very well conditioned: 
its distance to the closest matrix of a smaller rank is equal to $||U||=1$.


\begin{theorem}\label{thpert} (See \cite[Corollary 1.4.19]{S98} for $P= -M^{-1}E$.)
Suppose $M$ and $M+E$ are two nonsingular matrices of the same size
and $||M^{-1}E||  =\theta<1$. Then
$||I-(M+E)^{-1}M||  \le \frac{\theta}{1-\theta}$ and
$\||(M+E)^{-1}-M^{-1}||\le \frac{\theta}{1-\theta}||M^{-1}||$.
In particular $\||(M+E)^{-1}-M^{-1}||\le 0.5||M^{-1}||$
if $\theta\le 1/3$.
\end{theorem}

\subsection{The classes of Toeplitz, Hankel, Vandermonde and Cauchy
 matrices, some subclasses and factorizations, polynomial evaluation 
and interpolation}\label{sfour}

For larger integers $n$ the entries of  an $n\times n$ Van\-der\-monde matrix $V_{\bf s}$ 
 vary in magnitude greatly unless $|s_i|\approx 1$ for all $i$,
as is the case with 
 the  Van\-der\-monde matrices $\Omega$ and $\Omega^H$ below, which are
 unitary up to scaling by $\frac{1}{\sqrt n}$.

{\bf DFT and DFT-based matrices.} (See  \cite[Sections 1.2, 3.4]{BP94}.)
Write $\omega_n={\rm exp}(\frac{2\pi}{n}\sqrt{-1})$ to denote 
a primitive $n$th root of $1$. 
Its powers $1,\omega_n,\dots,\omega_n^{n-1}$
are equally spaced on
the unit circle $\{z:|z|=1\}$.
Let
$\Omega=\Omega_n=(\omega_n^{ij})_{i,j=0}^{n-1}$ denote the $n\times n$ matrix of 
{\em DFT}, that is  of the 
{\em discrete Fourier transform} at $n$ points. 
 $\Omega$ and $\Omega^H$
are quasiunitary,
whereas
  $\frac{1}{\sqrt n}\Omega$ and $\frac{1}{\sqrt n}\Omega^H$
and $\Omega^{-1}=\frac{1}{n}\Omega^H$
are unitary matrices,
 because $\Omega \Omega^H=nI$.
See, e.g., \cite[Sections 1.2 and 3.4]{BP94}
on a proof of the following theorem and on the numerical stability 
of the supporting algorithms.

\begin{theorem}\label{thdft}
For any vector ${\bf v}=(v_i)_{i=1}^{n}$
one can compute the vectors $\Omega{\bf v}$
and $\Omega^{-1}{\bf v}$ by using $O(n\log n)$
flops. If $n=2^k$ 
 is a power of 2, then
one can  compute the vectors $\Omega{\bf v}$
and $\Omega^{-1}{\bf v}$
by using
$0.5n\log_2 n$
 and 
$0.5n\log_2 n+n$ flops,
respectively.
\end{theorem}

\bigskip

{\bf Cauchy and Van\-der\-monde  matrices
and polynomial evaluation 
and interpolation.} (See Table \ref{t1} and \cite[Chapters 2 and 3]{P01}.)
It holds that
\begin{equation}\label{eqctr}
C_{\bf s,t}=-C_{\bf t,s}^T,
\end{equation}
\begin{equation}\label{eqfhr}
C_{\bf s,t}=\diag(t(s_i)^{-1})_{i=1}^{n}V_{\bf s}V^{-1}_{\bf t}\diag(t'(t_i))_{i=1}^{n}
\end{equation}
where ${\bf s}=(s_i)_{i=1}^{n}$, 
${\bf t}=(t_i)_{i=1}^{n}$, and
$t(x)=\prod_{i=0}^{n-1}(x-t_i)$.

Equation (\ref{eqfhr}) expresses a Cauchy matrix
$C_{\bf s,t}$ 
through the Van\-der\-monde matrix $V_{\bf s}$,
the inverse $V^{-1}_{\bf t}$ of the  Van\-der\-monde matrix
$V_{\bf t}$, the  
coefficients of
the auxiliary polynomial $t(x)$
defined by its roots
$t_0,\dots,t_{n-1}$, and
the values of this  polynomial
and its derivative $t'(x)$,  each at $n$ points.
Part (i) of the following simple theorem 
 states that {\em  polynomial multipoint evaluation 
and interpolation} with the knots $s_1,\dots,s_n$
are  equivalent to 
multiplication of the Van\-der\-monde  matrix
$V_{\bf s}$ by the coefficient vector of the polynomial
and the solution of the associated linear 
 system of equations with this matrix,
respectively.
Part (ii)  of the theorem shows 
shows equivalence of rational multipoint evaluation and interpolation to 
 the  similar equations for Cauchy (rather than Van\-der\-monde)  matrix.
Part (iii)  of the theorem shows  that the
reconstruction of the  polynomial coefficients
from the roots can be reduced to polynomial interpolation and consequently to 
solving a Van\-der\-monde
linear system of equations. 
on these links and similar links of rational
multipoint evaluation and interpolation to Cauchy matrices.
Together with equation (\ref{eqfhr}), the theorem
also links  
multipoint evaluation and interpolation 
for polynomials to the same tasks for
  rational functions (cf. \cite[Chapter 3]{P01}).

\begin{theorem}\label{thevint}
(i) Let $p(x)=\sum_{i=0}^{n-1}p_ix^i$, ${\bf p}=(p_i)_{i=0}^{n-1}$,
 ${\bf s}=(s_i)_{i=0}^{n-1}$, and ${\bf v}=(v_i)_{i=0}^{n-1}$ .
Then the equations $p(s_{i})=v_i$ hold for $i=0,1,\dots,n-1$
if and only if $V_{\bf s}{\bf p}={\bf v}$.
(ii) For a rational function $v(x)=\sum_{j=1}^n\frac{u_j}{x-t_j}$
with $n$ distinct poles $t_1,\dots,t_n$ and for $n$
distinct scalars $s_1,\dots,s_n$, write 
${\bf s}=(s_i)_{i=1}^n$, ${\bf t}=(t_j)_{j=1}^n$,
${\bf u}=(u_j)_{j=1}^n$, ${\bf v}=(v_i)_{i=1}^n$.
Then the equations $v_i=v(s_i)$, $i=1,\dots,n$ hold if and only if
$C_{\bf s,t}{\bf u}={\bf v}$.
(iii) The equation $\prod_{i=0}^{n-1}(x-t_i)=x^n+v(x)$, 
for $v(x)=\sum_{i=0}^{n-1}t_ix^i$ and for $n$ distinct knots $t_0,\dots,t_{n-1}$,
is equivalent
to the linear system of $n$ equations,
$v(t_i)=-t_i^n$ for $i=0,\dots,n-1$.
\end{theorem}

\begin{theorem}\label{thcv} 
(i) $\det (V)=\prod_{i<k}(s_i-s_k)$ and 
$\det (C)=\prod_{i<j}(s_i-s_j)(t_i-t_j)/\prod_{i,j}(s_i-t_j)$,
and so
the matrices $V$ and $C$ of Table \ref{t1}
are nonsingular where all scalars $s_1,\dots,s_n,t_1,\dots,t_n$
are distinct. (ii) A row interchange preserves both Van\-der\-monde	  
and Cauchy structures. A column interchange preserves 	  
Cauchy structure. 
\end{theorem}

Next we will specify a subclass of Cauchy matrices most closely linked 
 to Van\-der\-monde and transposed Van\-der\-monde
 matrices (cf. Definition \ref{defcvdft}). 
At first
write
\begin{equation}\label{eqvf}
V_f=((f\omega_n^{i-1})^{j-1})_{i,j=1}^{n}=\Omega\diag(f^{j-1})_{j=1}^{n},
\end{equation}
\begin{equation}\label{eqcef}
C_{{\bf s},f}=\Bigg(\frac{1}{s_i-f\omega_n^{j-1}}\Bigg)_{i,j=1}^{n},~
C_{e,{\bf t}}=\Bigg(\frac{1}{e\omega_n^{i-1}-t_j}\Bigg)_{i,j=1}^{n},~
C_{e,f}=\Bigg(\frac{1}{e\omega_n^{i-1}-f\omega_n^{j-1}}\Bigg)_{i,j=1}^{n}
\end{equation}
for two distinct scalars $e$ and $f\neq 0$. Then observe that $\Omega=V_1=V_{\bf s}$ 
for ${\bf s}=(\omega_n^{i-1})_{i=1}^{n}$,
$\Omega^H=V_{\bf t}$ for ${\bf t}=(\omega_n^{1-i})_{i=1}^{n}$
(so both $\Omega$ and $\Omega^H$are Van\-der\-monde matrices),
 $C_{{\bf s},f}=C_{\bf s,t}$,
$C_{e,{\bf t}}=C_{\bf s,t}$,
$C_{e,f}=C_{\bf s,t}$ where ${\bf s}=(e\omega_n^{j-1})_{i,j=1}^{n}$ and/or ${\bf t}=(f\omega_n^{j-1})_{i,j=1}^{n}$,
and {\em the matrices} $V_f$ {\em are quasiunitary} where $|f|=1$.

Now let ${\bf t}=(f\omega_n^{j-1})_{j=1}^{n}$, 
obtain $t(x)=x^n-f^n$, $t(x)=nx^{n-1}$, $t(s_i)=s_i^n-f^n$,
$t'(t_i)=nf^{n-1}\omega_n^{1-i}$  for all $i$,
and 
$nV_f^{-1}=\diag(f^{1-i})_{i=1}^{n}\Omega^H$,
substitute into (\ref{eqfhr}), 
and obtain

\begin{equation}\label{eqfhrf}
C_{{\bf s},f}=
\diag\Bigg(\frac{f^{n-1}}{s_i^{n}-f^n}\Bigg)_{i=1}^{n}V_{\bf s}\diag(f^{1-i})_{i=1}^{n}\Omega^H\diag(\omega_n^{1-i})_{i=1}^{n}.
\end{equation}
If in addition ${\bf s}=(e\omega_n^{i-1})_{i=1}^{n}$,
then $s_i^n=e^n$ for all $i$ and $V_{\bf s}=V_f$.
Substitute into (\ref{eqfhrf}) and obtain

\begin{equation}\label{eqfhref}
C_{e,f}=n\frac{f^{n-1}}{e^n-f^{n}}\Omega\diag((e/f)^{i-1})_{i=1}^{n}\Omega^H\diag(\omega^{1-i})_{i=1}^{n}.
\end{equation}

\begin{definition}\label{defcvdft} 
Hereafter we refer to 
the matrices 
$V_f$, $C_{{\bf s},f}$, $C_{e,{\bf t}}$, and $C_{e,f}$
for all scalars $e$ and $f$
as {\em FV, FC, CF}, and {\em FCF matrices}, respectively. We refer  to
the matrices 
$C_{{\bf s},f}$ and $C_{e,{\bf t}}$ as
{\em CV matrices}
and to 
 the FV matrices $V_f$ and the FCF matrices $C_{e,f}$
as
the {\em DFT-based} matrices.
\end{definition}
Equations (\ref{eqvf}) and (\ref{eqfhref}) link 
the
DFT-based matrices  to the DFT matrix $\Omega$.
Similarly to
this matrix
they have
their basic sets of knots
$\mathbb S=\{s_1,\dots,s_n\}$
and
$\mathbb T=\{t_1,\dots,t_n\}$
equally spaced on the unit circle $\{z:~|z|=1\}$.
Equations (\ref{eqfhrf}) link the
 CV matrices
to Vandermonde matrices $V_{\bf s}$
and  $V_{\bf t}$, respectively.
Combine  equation (\ref{eqvf}) and (\ref{eqfhref}) with Theorem \ref{thdft}
 to obtain the following results.

\begin{theorem}\label{thcvdft} 
$O(n\log n)$  flops are sufficient to compute the product $M{\bf f}$
of a DFT-based Van\-der\-monde or Cauchy $n\times n$ matrix $M$ and a vector
${\bf f}$. If the matrix $M$ is nonsingular,
then
$O(n\log n)$  flops are also sufficient to compute the
solution ${\bf x}$ to a linear system of $n$ equations $M{\bf x}={\bf f}$.
\end{theorem}

\noindent{\bf $f$-circulant matrices.}
$Z_f=\begin{pmatrix} {\bf 0}^T & f\\ I_{n-1} & {\bf 0}\end{pmatrix}$
is the  $n\times n$ matrix of $f$-circular shift
for a scalar $f$,
\begin{equation}\label{eqrever}
JZ_fJ=Z_{f}^T,~~
JZ_f^TJ=Z_{f} 
\end{equation}
for any pairs of scalars $e$ and $f$, and if $f\neq 0$, then
\begin{equation}\label{eqinv}
Z_f^{-1}=Z_{1/f^T}.
\end{equation}

$Z_f({\bf v})=\sum_{i=1}^{n}v_{i}Z_f^{i-1}$
is an  $f$-circulant matrix,
called circulant for $f=1$.
It is a Toeplitz matrix defined by its first column ${\bf v}=(v_i)_{i=1}^n$
and by a scalar $f\neq 0$. It can be called a {\em DFT-based Toeplitz matrix}
in view of the following results.

\begin{theorem}\label{thcpw} (See \cite{CPW74}.)
We have 
$Z_1({\bf v})=\Omega^{-1}D(\Omega{\bf v})\Omega.$
More generally, for any $f\ne 0$, we have
$Z_{f^n}({\bf v})=V_f^{-1}D(V_f{\bf v})V_f$
where
$\Omega=(\omega_n^{ij})_{i,j=0}^{n-1}$ is the $n\times n$ matrix of DFT,
$D({\bf u})=\diag(u_i)_{i=0}^{n-1}$ for a vector ${\bf u}=(u_i)_{i=0}^{n-1}$  
and the matrix $V_f=\Omega\diag(f^{i})_{i=0}^{n-1}$ of (\ref{eqvf}). 
\end{theorem}

\bigskip

{\bf The complexity of computations with Toep\-litz, Hankel,
Cauchy and Van\-der\-monde  matrices.}
Theorems \ref{thdft} and \ref{thcpw} combined support
numerically stable  computation of the product 
by a vector
of an
$f$-circulant matrix 
$Z_e({\bf u})$ (as well as of its inverse if the matrix is nonsingular)
by using $O(n\log n)$ flops.
We can extend this cost bound to multiplication of
 a Toep\-litz matrix $T$ of Table \ref{t1} 
 by a vector,
by embedding the matrix
into a $2^k\times 2^k$ circulant matrix
for  $k=\lceil \log_2(2n-1)\rceil$.
Each of pre- and post-multiplication by the matrix $J$, that is the
cyclic interchange of rows or columns, transforms a Toep\-litz matrix into
a Hankel matrix and vice versa,  
and therefore transforms accordingly the algorithms for 
matrix inversion and solving a linear systems of equations.

Numerically unstable algorithms using nearly linear number of flops
(namely $O(n\log^2n)$ flops) are known for multiplying
 general Van\-der\-monde and Cauchy $n\times n$ matrices 
by a vector and solving Toeplitz, Hankel,  
Van\-der\-monde and Cauchy
nonsingular linear systems of $n$ equations (cf. \cite[Chapter 2 and 3]{P01}).
Numerically stable  known  algorithms for all
these problems run in quadratic arithmetic time,
except that 
numerically stable  algorithms of
\cite{MRT05}, \cite{CGS07}  and \cite{XXG12}
approximate the solution of 
nonsingular Toeplitz linear systems
in nearly linear time.
We seek extension of the latter algorithms to
 Van\-der\-monde and Cauchy 
computations.

\section{The structures of Toeplitz, 
Hankel, Van\-der\-monde and Cauchy
types. Displacement ranks and generators}\label{sfourd}

We generalize
 the four classes of
matrices of Table \ref{t1}
by employing
the Sylvester displacements
$AM-MB$ where 
the pair of {\em operator matrices} $A$ and $B$
is associated with a fixed matrix structure.  
(See \cite[Theorem 1.3.1]{P01} on a simple link to the Stein displacements  
$M-AMB$.)
The rank, the $\tau$-rank, 
generators, and $\tau$-generators of the displacement
of a matrix $M$ (for a fixed  operator matrices $A$ and $B$
and  tolerance $\tau$)
are called {\em displacement rank} (denoted $d_{A,B}(M)$),
$\tau$-{\em displacement rank},
  {\em displacement generator},
 and  $\tau$-{\em displacement generator}
of the matrix $M$,
respectively (cf. \cite{KKM79}), \cite{P01}, \cite{BM01}).
In our Theorems \ref{thdsplr} and \ref{thdexpr}  we write 
$(t), ~(h),~ (th),~ (v), ~(v^T)$, and $(c)$ to indicate the matrix structures of
Toeplitz, 
Hankel, 
Toeplitz or 
Hankel,
 Van\-der\-monde, transposed Van\-der\-monde, and Cauchy
types, respectively,
which we define next.

\begin{definition}\label{def4t}
If the displacement rank of a matrix is small 
(in context) for a pair of 
operator matrices 
associated with 
Toeplitz, Hankel, Van\-der\-monde, transpose of  Van\-der\-monde or Cauchy matrices
in Theorem \ref{thdsplr}
below,
then the matrix is said to have the {\em structure
of Toeplitz, Hankel,  Van\-der\-monde,  transposed Van\-der\-monde or Cauchy type},
respectively. Hereafter 
 $\mathcal T$, 
$\mathcal H$,
$\mathcal V$,
$\mathcal V^T$,
 and
$\mathcal C$ denote the five classes of these
 matrices (cf. Table \ref{taboperm}).
The classes $\mathcal V$,
$\mathcal V^T$,
 and
$\mathcal C$ consist of distinct subclasses 
$\mathcal V_{\bf s}$,
$\mathcal V^T_{\bf s}$,
 and
$\mathcal C_{\bf s,t}$
defined by the vectors ${\bf s}$
and  ${\bf t}$ and the operator matrices
$D_{\bf s}$ and
$D_{\bf t}$, respectively,
or equivalently by the  bases $V_{\bf s}$
and
$C_{\bf s,t}$ of these subclasses.
To simplify the notation
we will sometimes drop the subscripts
${\bf s}$ and ${\bf t}$
where they are not important 
or are defined by context.
\end{definition}

\begin{definition}\label{def4tf} (Cf. Definition \ref{defcvdft}.)
 In the case where the vectors  ${\bf s}$ 
and  ${\bf t}$ turn into the vectors 
$e(\omega_n^{i-1})_{i=1}^n$ and
$f(\omega_n^{i-1})_{i=1}^n$
for some scalars $e$ and $f$, we 
define the matrix classes
$\mathcal {FV}=
\cup_e\mathcal V_e$,
$\mathcal {FC}=
\cup_f\mathcal C_{{\bf s},f}$,
$\mathcal {CF}=
\cup_e\mathcal C_{e,{\bf t}}$,
and $\mathcal {FCF}
=\cup_{e,f}\mathcal C_{e,f}$
where the unions are over all complex scalars $e$ and $f$.
These matrix classes extend
 the classes of FV, FC, CF, and FCF matrices, 
respectively. We also define the classes $\mathcal {CV}$
(extending the CV matrices) 
and
$\mathcal {V^TF}=\cup_e\mathcal V_e^T$.
We say that they consist of  $FV$-like,
$V^TF$-like, $FC$-like, $CF$-like, $FCF$-like, and $CV$-like
 matrices, which have structures
of $\mathcal {FV}$-type, $\mathcal {V^TF}$-type,
$\mathcal {FC}$-type, $\mathcal {CF}$-type,
 $\mathcal {FCF}$-type, and  $\mathcal {CV}$-type,
respectively.
\end{definition}

One can readily verify the following results.
\begin{theorem}\label{thdsplr} {\em Displacements of basic structured matrices.}

(th) For a pair of scalars $e$ and $f$ and
 two matrices $T$ (Toeplitz)
and $H$ (Hankel) of Table \ref{t1},
 the following displacements have ranks at most $2$
(see some expressions for the shortest displacement 
generators in \cite[Section 4.2]{P01}),
$$Z_eT-TZ_f,~
Z_e^TT-TZ_f^T,~Z_e^TH-HZ_f~{\rm and}~Z_eH-HZ_f^T.$$ 

(v) For a scalar $e$ and a
Vandermonde matrix $V$ of Table \ref{t1} we have 
\begin{equation}\label{eqvand}
VZ_e=D_{\bf s}V-(s_i^n-e)_{i=1}^n{\bf e}_n^T,
\end{equation}
\begin{equation}\label{eqvandt}
Z_e^TV^T=V^TD_{\bf s}-{\bf e}_n((s_i^n-e)_{i=1}^n)^T.
\end{equation}
Consequently
the displacements 
$D_{\bf s}V-VZ_{e}$ and $Z_{e}^TV^T-V^TD_{\bf s}$
have rank at most $1$ and
vanish if
\begin{equation}\label{eqrz}
s_i^n=e~{\rm for}~i=1,\dots,n.
\end{equation}

\medskip

(c) For 
 two vectors 
${\bf s}=(s_i)_{i=1}^n$ and
${\bf t}=(t_i)_{i=1}^n$ having $2n$
distinct components, 
 a 
 Cauchy matrix $C$ of Table \ref{t1},
and the vector ${\bf e}=(1,\dots,1)^T$  of dimension $n$, filled with ones,
we have
\begin{equation}\label{eqcch}
D_{\bf s}C-CD_{\bf t}={\bf e}{\bf e}^T,
~~\rank(D_{\bf s}C-CD_{\bf t})=1.
\end{equation}
\end{theorem}


The 
following theorem shows that
variation of the scalars
$e$ and $f$, defining 
the operator matrices $Z_e$ and $Z_f$,
 makes negligible impact on the matrix  structure,
and so the classes  $\mathcal T$, 
$\mathcal H$,
$\mathcal V$, and
$\mathcal V^T$ 
do not depend on the choice of these scalars. 

\begin{theorem}\label{thzef} 
For two scalars $e$ and $f$
and five matrices $A$, $B$, $C$, $D$, and $M$ we have
$d_{C,D}(M)-d_{A,B}(M)\le 1$ 
where either $A=C$, $B=Z_e$, $D=Z_f$ or $A=C$, $B=Z_e^T$, $D=Z_f^T$
and similarly where either 
 $B=D$, $A=Z_e$, $C=Z_f$ or $B=D$, $A=Z_e^T$, $C=Z_f^T$.
\end{theorem} 
\begin{proof} 
The matrix $Z_b-Z_c=(b-c){\bf e}_1{\bf e}_n^T$ 
has rank at most $1$
for any pair of scalars $b$ and $c$.
Therefore the matrices 
$(Z_bM-MB)-(Z_cM-MB)=Z_bM-Z_cM=(Z_b-Z_c)M$ and 
$(AZ-MZ_b)-(AZ-MZ_c)=-M(Z_b-Z_c)$
have ranks at most $1$.
\end{proof}

Table \ref{taboperm} displays the pairs of operator matrices
associated with the matrices of the seven classes $\mathcal T$, 
$\mathcal H$,
$\mathcal V_{\bf s}$,
$\mathcal V^{-1}_{\bf s}$,
$\mathcal V^{T}_{\bf s}$,
$\mathcal V^{-T}_{\bf s}$,
and
$\mathcal C_{\bf s,t}$.
Five of these classes  are employed in Theorems \ref{thdsplr} and \ref{thdexpr}.
 $\mathcal V^{-1}_{\bf s}$
and
$\mathcal V^{-T}_{\bf s}$
denote 
the 
classes of 
 the inverses and the 
 transposed
inverses 
of the matrices
of the class $\mathcal V_{\bf s}$, respectively.
We obtain 
the pairs of their associated operator matrices by interchanging the matrices in the
pairs of the operator matrices for 
the classes $\mathcal V_{\bf s}$ and $\mathcal V^T_{\bf s}$,
respectively (see equation (\ref{eqinvm}) of the next section). 

The following theorem  
 expresses the $n^2$ entries of an $n\times n$ matrix $M$
 through the $2dn$ entries of its displacement 
generator $(F,G)$ defined
under the operator matrices of Theorem \ref{thdsplr}
and Table \ref{taboperm}.
See some of these and other expressions  
for various classes  of structured matrices 
through their generators
in 
\cite{GO94}, \cite[Sections 4.4 and 4.5]{P01}, and 
 \cite{PW03}.

\begin{theorem}\label{thdexpr}
Suppose $s_1,\dots,s_n,t_1,\dots,t_n$ are  $2n$
distinct scalars, 
${\bf s}=(s_k)_{k=1}^n$, ${\bf t}=(t_k)_{k=1}^n$,
$V=(s_i^{k-1})_{i,k=1}^n$,   $C=(\frac{1}{s_i-t_k})_{i,k=1}^n$,
$e$ and $f$ are two distinct scalars,
${\bf f}_1, \dots, {\bf f}_d,{\bf g}_1, \dots, {\bf g}_d$ are $2d$ vectors of dimension $n$,
 ${\bf u}_1, \dots, {\bf u}_n,{\bf v}_1, \dots, {\bf v}_n$ are $2n$ vectors of dimension $d$,
and $F$ and $G$ are $n\times d$ matrices such that  
$F=\begin{pmatrix}{\bf u}_1\\ \vdots\\ {\bf u}_n\end{pmatrix}=({\bf f}_1~|~\cdots~|~{\bf f}_d),~~
G=\begin{pmatrix}{\bf v}_1\\ \vdots\\ {\bf v}_n\end{pmatrix}=({\bf g}_1~|~\cdots~|~{\bf g}_d)$.
 Then 

(t) $(e-f)M=\sum_{j=1}^dZ_e({\bf f}_j)Z_f(J{\bf g}_j)$
if $Z_eM-MZ_f=FG^T$, $e\neq f$;

$(e-f)M=\sum_{j=1}^dZ_e(J{\bf f}_j)^TZ_f({\bf g}_j)^T=J\sum_{j=1}^dZ_e(J{\bf f}_j)Z_f({\bf g}_j)J$
if $Z_e^TM-MZ_f^T=FG^T$, $e\neq f$,

(h) $(e-f)M=\sum_{j=1}^dZ_e({\bf f}_j)Z_f({\bf g}_j)J$
if $Z_eM-MZ_f^T=FG^T$,  $e\neq f$; 

$(e-f)M=J\sum_{j=1}^dZ_e(J{\bf f}_j)Z_f(J{\bf g}_j)^T$
if $Z_e^TM-MZ_f=FG^T$,  $e\neq f$,

(v) $M=
\diag(\frac{1}{s_i^n-e})_{i=1}^n\sum_{j=1}^d\diag({\bf f}_j)VZ_{e}(J{\bf g}_j)$
if  
$D_{\bf s}M-MZ_{e}=FG^T$
and if
$s_i^n\neq e$ for $i=1,\dots,n$;  

(v$^T$) $M=\diag(\frac{1}{e-s_i^n})_{i=1}^n\sum_{j=1}^dZ_{e}(J{\bf f}_j)^TV^T\diag({\bf g}_j)$
if  
$Z_{e}^TM-MD_{\bf s}=FG^T$ 
 and if $s_i^n\neq e$ for $i=1,\dots,n$;


(c) $M=\sum_{j=1}^d\diag({\bf f}_j)C\diag({\bf g}_j)=\left(\frac{{\bf u}_i^T{\bf v}_j}{s_i-t_j}\right)_{i,j=0}^{n-1}$
if
$D_{\bf s}M-MD_{\bf t}=FG^T$. 
\end{theorem}
\begin{proof}
Parts $(t)$ and $(h)$ are taken from \cite[Examples 4.4.2 and 4.4.4]{P01}.
Part $(c)$ is taken from \cite[Example 1.4.1]{P01}.
To prove part ($v$), combine the equations
$D_{\bf t}M-MZ_{e}=FG^T$
and $Z_eZ_{1/e}^T=I$
(cf. (\ref{eqinv}))  and
deduce 
that 
$M-D_{\bf t}MZ_{1/e}^T=-F(Z_{1/e}G)^T$.
Then  obtain from
 \cite[Example 4.4.6 (part b)] {P01} 
that
$M=
e\diag(\frac{1}{t_i^n-e})_{i=1}^n\sum_{j=1}^n\diag({\bf f}_j)VZ_{1/e}(Z_{1/e}{\bf g}_j)^T$.
Substitute $eZ_{1/e}(Z_{1/e}{\bf g}_j)=Z_e(J{\bf g}_j)^T$ and obtain 
the claimed  expression of part ($v$).
Next  transpose the equation
$Z_{e}^TM-MD_{\bf t}=FG^T$
and yield
$D_{\bf t}M^T-M^TZ_{e}=-GF^T$.
 From 
 part $(v)$ 
obtain
$M^T=\diag(\frac{1}{e-t_i^n})_{i=1}^n\sum_{j=1}^n \diag({\bf g}_j)VZ_{e}(J{\bf f}_j)$.
Transpose this equation
and arrive at 
part $(v^T)$.
\end{proof}


\begin{table}
\caption{Operator matrices for the seven classes $\mathcal T$, 
$\mathcal H$,
$\mathcal V_{\bf s}$,
$\mathcal V^{-1}_{\bf s}$, 
$\mathcal V^{T}_{\bf s}$,
$\mathcal V^{-T}_{\bf s}$,
and
$\mathcal C_{\bf s,t}$} 
\label{taboperm}
\begin{center}
\renewcommand{\arraystretch}{1.2}
\begin{tabular}{|c|c|c|c|c|c|c|}
 \hline
 $\mathcal T$ & $\mathcal H$   & $\mathcal V_{\bf s}$ 
&$\mathcal V^{-1}_{\bf s}$
&$\mathcal V^T_{\bf s}$ 
& $\mathcal V^{-T}_{\bf s}$ &$\mathcal C_{\bf s,t}$\\ \hline 
$(Z_e,Z_f)$ & $(Z_e^T,Z_f)$  & $(D_{\bf s},Z_e)$ &  $(Z_e,D_{\bf s})$ &  
$(Z_e^T,D_{\bf s})$  &  $(D_{\bf s},Z_e^T)$&  $(D_{\bf s},D_{\bf t})$\\ \hline 
 $(Z_e^T,Z_f^T)$ & $(Z_e,Z_f^T)~$ & & & & &  \\ \hline
\end{tabular}
\end{center}
\end{table}

 By combining the estimates of the previous section for 
the cost of multiplication by a vector
of Toeplitz, Hankel, Van\-der\-monde, transpose of  Van\-der\-monde and Cauchy matrices
with 
Theorem \ref{thdexpr}
we obtain the following results.

\begin{theorem}\label{thmbv} 
Given a vector ${\bf v}$ of a dimension $n$,
one can compute the product $M{\bf v}$
 by using $O(dn\log n)$ flops for an $n\times n$
matrix $M$ 
in the classes $\mathcal T$ or 
$\mathcal H$
and by using $O(dn\log^2 n)$ flops for an $n\times n$
matrix $M$ in $\mathcal V$, $\mathcal V^T$, or
$\mathcal C$.
\end{theorem}

\begin{remark}\label{reexpr}
By virtue of Theorem \ref{thdexpr}
the displacement operators $M\rightarrow AM-MB$
are nonsingular provided
 that $e\neq f$ in parts $(t)$ and $(h)$
and that 
$t_i^n\neq e$ for $i=1,\dots,n$ in parts $(v)$ and $(v^T)$.
We can  apply Theorem \ref{thzef} 
to satisfy these assumptions.
\end{remark}

\begin{remark}\label{reperm} (Cf. part $(ii)$ of Theorem \ref{thcv}.)
Parts $(v)$ and $(c)$ of Theorem  \ref{thdexpr} imply that
a row interchange preserves
the matrix structures of the Vandermonde and Cauchy types, 
whereas 
a  column interchange preserves
the matrix structures of the
transposed  Vandermonde and Cauchy types. 
\end{remark}

\section{Matrix operations in terms of displacement generators}\label{soper}

We can pairwise
 multiply and invert
structured matrices faster
if we express the inputs and the intermediate and final results 
of the computations
through short displacement generators
rather than the matrix entries. Such computations are
possible by virtue of the following simple results
from \cite{P00} and \cite[Section 1.5]{P01}
(extending  \cite{P90}).

\begin{theorem}\label{thdgs}
Assume five matrices $A$, $B$, $C$, $M$ and $N$
and a pair of scalars $\alpha $ and $\beta $.
Then as long as the matrix sizes are compatible
we have
\begin{equation}\label{eqlcm}
A(\alpha M+\beta N)-(\alpha M+\beta N)B=\alpha (AM-MB)+\beta (AN-NB),
\end{equation}
\begin{equation}\label{eqtrm}
A^TM^T-B^TM^T=-(BM-MA)^T,
\end{equation}
\begin{equation}\label{eqprm}
A(MN)-(MN)C=(AM-MB)N+M(BN-NC).
\end{equation}
Furthermore for a nonsingular  matrix $M$  we have
\begin{equation}\label{eqinvm}
AM^{-1}-M^{-1}B=-M^{-1}(BM-MA)M^{-1}.
\end{equation}
\end{theorem}

\begin{corollary}\label{codgdr}
For five matrices $A$, $B$, $F$, $G$, and $M$
of sizes $m\times m$, $n\times n$,
$m\times d$, $n\times d$, and $m\times m$, 
respectively,
let us write $F=F_{A,B}(M)$, $G=G_{A,B}(M)$,
 and  $d=d_{A,B}(M)$
if $AM-MB=FG^T$.
Then under the assumptions of Theorem \ref{thdgs}
 we have
\begin{eqnarray*}
F_{A,B}(\alpha M+\beta N)=(\alpha F_{A,B}(M)~|~\beta F_{A,B}(N)),\\
G_{A,B}(\alpha M+\beta N)=(G_{A,B}(M)~|~G_{A,B}(N)),  \\
F_{A,B}(M^T)=-G_{B^T,A^T}(M^T),~G_{A,B}(M^T)=F_{B^T,A^T}(M^T), \\
F_{A,C}(MN)=(F_{A,B}(M)~|~M F_{B,C}(N)), \\
G_{A,C}(MN)=(N^TG_{A,B}(M)~|~G_{B,C}(N)),  \\
F_{A,B}(M^{-1})=-M^{-1}G_{B,A}(M),~G_{A,B}(M^{-1})=M^{-1}F_{B,A}(M).
\end{eqnarray*}

\noindent Consequently
\begin{eqnarray*}
d_{A,B}(\alpha M+\beta N)\le d_{A,B}(M)+d_{A,B}(N), \\
d_{A,B}(M^T)=d_{B^T,A^T}(M),\\
d_{A,C}(MN)\le d_{A,B}(M)+d_{B,C}(N),  \\
d_{A,B}(M^{-1})=d_{B,A}(M).
\end{eqnarray*}
\end{corollary}

The corollary and Theorem \ref{thdexpr} 
together reduce the inversion of a nonsingular $n\times n$ matrix $M$ 
given by its displacement generator of a length $d$
to solving
$2d$ linear systems of equations with this coefficient matrix $M$,
rather than the $n$ linear systems $M{\bf x}_i={\bf e}_i$, $i=1,\dots,n$ for general matrix $M$.

Given short displacement generators for 
the matrices $M$ and $N$,
we can apply Corollary \ref{codgdr}
and readily express short displacement generators for 
the matrices $M^T$, $\alpha M+\beta N$, and $MN$
through the matrices $M$ and $N$ and their
 displacement generators, 
but the expressions 
for the displacement generator of
the inverse $M^{-1}$
involve the inverse  itself.

Some fast inversion algorithms
such as the di\-vide-and-con\-quer MBA algorithm
of \cite{M80} and \cite{BA80}
involve auxiliary matrices with displacement generators 
whose lengths exceed the displacement rank
shared by the input and output matrices.
If uncontrolled, the di\-vide-and-con\-quer process 
can blow up 
the length
of the displacement generators 
and the computational cost (see \cite[Chapter 5]{P01}),
but one can apply Theorem \ref{thcompr} to
recompress the generators and obtain the following result.
 
\begin{corollary}\label{coinv} \cite{M80}, \cite{BA80}.
The MBA algorithm computes a displacement
generator of
the inverse of a nonsingular $n\times n$ matrix $M$ 
by using $O(d^2n\log^s n)$ flops
where $s=2$ if $M$ is from the class
$\mathcal T$ or 
$\mathcal H$ and  $s=3$
if $M$ is from the class
$\mathcal V$,
$\mathcal V^T$, or 
$\mathcal C$.
\end{corollary}


\section{Transformation of Matrix Structures}\label{sdtr}


\subsection{Maps and multipliers}\label{shld}

Recall that 
each of 
the five matrix classes
 $\mathcal T$, 
$\mathcal H$,
$\mathcal V$,
$\mathcal V^T$, 
and
$\mathcal C$
consists of the
 matrices $M$ 
whose displacement rank, $\rank(AM-MB)$ 
is small (in context) for 
a pair of operator matrices $(A,B)$ associated with this class.
Thus these pairs represent the structure  of
the  matrix classes.

 Theorem \ref{thdgs} shows the impact of 
elementaty matrix operations  on the associated operator matrices
$A$ and $B$. 
For linear combinations,   
transposes and inverses 
the original  pair $(A,B)$ either stays invariant or changes 
into $(-B^T,A^T)$ or $(-B,A)$,
respectively.
If the inputs are in any of the classes   
$\mathcal T$, 
$\mathcal H$, and 
$\mathcal C_{\bf s,t}$, then so are the outputs,
whereas the transposition maps the classes 
$\mathcal V$ 
and $\mathcal V^T$  into one another 
and the inversion maps them
into the classes
$\mathcal V^{-1}$ 
and $\mathcal V^{-T}$,
respectively.
The  impact of multiplication on 
matrix structure is quite different.
As we can see from (\ref{eqprm}) and  Table \ref{taboper},
the map $M\rightarrow PMN$ can imply transition from the associated pair of operator matrices
 $(A,B)$
 to any new pair $(C,D)$ of our choice,
that is we can transform
the matrix structures of the five classes into each other  at will.
 The following theorem and Table \ref{tabmsmp}
specify such  structure transforms given by the maps 
$M\rightarrow MN$,  $N\rightarrow MN$, and $M\rightarrow PMN$
for appropriate 
 multipliers $P$, $M$, and $N$. 

\begin{table}
\caption{Operator matrices for matrix product} 
\label{taboper}
\begin{center}
\renewcommand{\arraystretch}{1.2}
\begin{tabular}{|c|c|c|c|}
\hline 
$P$ & $M$ & $N$ & $PMN$ \\ \hline 
$C$ & $A$ & $B$ &  $C$ \\  \hline 
$A$ & $B$ & $D$  & $D$ \\ \hline 
\end{tabular}
\end{center}
\end{table}

\begin{theorem}\label{thmap} 
It holds that 

(i) $MN\in \mathcal T$ if the  pair of matroces $(M,N)$ is in any
of the  pairs of matrix classes
$(\mathcal T,\mathcal T)$,
$(\mathcal H,\mathcal H)$,
$(\mathcal V^{-1}_{\bf s},\mathcal V_{\bf s})$ and $(\mathcal V^{T}_{\bf s},\mathcal V^{-T}_{\bf s})$, 

(ii) $MN\in \mathcal H$ if the  pair $(M,N)$ is in any
of the pairs $(\mathcal T,\mathcal H)$,
$(\mathcal H,\mathcal T)$,
$(\mathcal V^{-1}_{\bf s},\mathcal V_{\bf s}^{-T})$ and $(\mathcal V^{T}_{\bf s},\mathcal V_{\bf s})$,

(iii) $MN\in \mathcal V_{\bf s}$ if
the  pair 
$(M,N)$ is in any
of the  pairs $(\mathcal V_{\bf s},\mathcal T)$,
$(\mathcal V^{-T}_{\bf s},\mathcal H)$, and $(\mathcal C_{\bf s,t},\mathcal V_{\bf t})$, 

(iv) $MN\in \mathcal V^T_{\bf s}$ if
the  pair 
$(M,N)$ is in any
of the  pairs  $(\mathcal T,\mathcal V^T_{\bf s})$,
$(\mathcal H,\mathcal V^{-1}_{\bf s})$ and $(\mathcal V^{T}_{\bf q},\mathcal C_{\bf q,s})$, 

(v) $MN\in \mathcal C_{\bf s,t}$ if
the  pair 
$(M,N)$ is in any
of the pairs  $(\mathcal C_{\bf s,q},\mathcal C_{\bf q,t})$,
$(\mathcal V^{-T}_{\bf s},\mathcal V^T_{\bf s})$ and $(\mathcal V_{\bf s},\mathcal V^{-1}_{\bf s})$,

(vii) $PMN\in \mathcal C_{\bf s,t}$ if
the  triple
$(M,N,P)$ is in any
of the  triples  $(\mathcal V_{\bf s},\mathcal H,\mathcal V^{T}_{\bf t})$ and 
$(\mathcal V^{-T}_{\bf s},\mathcal H,\mathcal V^{-1}_{\bf t})$. 
\end{theorem}

\begin{table}
\caption{Mapping matrix structures by means of multiplication} 
\label{tabmsmp}
\begin{center}
\renewcommand{\arraystretch}{1.2}
\begin{tabular}{|c|c|c|c|c|}
 \hline
$\mathcal T$ & $\mathcal H $ & $ \mathcal V_{\bf s}$& 
$\mathcal V^T_{\bf s}$ & $\mathcal C_{\bf s,t}$\\ \hline 
$\mathcal T\mathcal T$, $\mathcal V^{T}_{\bf s}\mathcal V^{-T}_{\bf s}$ & 
$\mathcal T\mathcal H$,  $\mathcal V^T_{\bf s}\mathcal V_{\bf s}$ & $\mathcal V_{\bf s}\mathcal T$, 
$\mathcal C_{\bf s,t}\mathcal V_{\bf t}$ &  $\mathcal T\mathcal V^T_{\bf s}$, 
$\mathcal H\mathcal V^{-1}_{\bf s}$  &  
$\mathcal V^{-T}_{\bf s}\mathcal V^T_{\bf t}$, 
$\mathcal V_{\bf s}\mathcal V^{-1}_{\bf t}$, $\mathcal C_{\bf s,q}\mathcal C_{\bf q,t}$\\  \hline 
~$\mathcal V^{-1}_{\bf s}\mathcal V_{\bf s}$, $\mathcal H\mathcal H$ & 
$\mathcal H\mathcal T$, $\mathcal V^{-1}_{\bf s}\mathcal V^{-T}_{\bf s}$  & 
$\mathcal V^{-T}_{\bf s}\mathcal H$ & 
$\mathcal V^T_{\bf q}\mathcal C_{\bf q,s}$ & $\mathcal V_{\bf s}\mathcal H\mathcal V^T_{\bf t}$, 
$\mathcal V^{-T}_{\bf s}\mathcal H\mathcal V^{-1}_{\bf t}$\\  \hline
\end{tabular}
\end{center}
\end{table}


The maps of Theorem \ref{thmap} and Table \ref{tabmsmp} 
hold for any choice of the multipliers $P$ and $N$
from the indicated  classes.
To simplify the computation
of the products
$MN$, $MP$ and $MNP$,
 we can choose  
the multipliers  $J$, $V_{\bf r}$, $V_{\bf r}^T$, 
$V_{\bf r}^{-1}$, $V_{\bf r}^{-T}$,
 and
$C_{\bf p,r}$,
all
having displacement rank 1,
to represent the classes 
 $\mathcal H$, $\mathcal V_{\bf r}$,
 $\mathcal V_{\bf r}^T$,
$\mathcal V_{\bf r}^{-1}$, $\mathcal V_{\bf r}^{-T}$,
  and
$\mathcal C_{\bf p,r}$
of the table and the theorem, respectively,
where ${\bf p}$ and ${\bf r}$ can stand for ${\bf q}$,  ${\bf s}$, and ${\bf t}$.
Hereafter we call this choice of multipliers {\em canonical}.
We call them {\em canonical and DFT-based} if 
up to the factor $J$
they are also DFT-based,
that is if ${\bf p}$ and ${\bf r}$ 
are of the form $f(\omega_n^{i-1})_{i=1}^{n}$. 
These multipliers are quasiunitary 
where $|f|=1$.
By combining 
Corollary \ref{codgdr} and
Theorem
 \ref{thmap}
we obtain the following
result.

\begin{corollary}\label{codcan} 
Given a
displacement generator of 
a length $d$ for
an $n\times n$ matrix $M$
of any of the classes 
$\mathcal T$, 
$\mathcal H$,
$\mathcal V$,
$\mathcal V^T$, and
$\mathcal C$,
$O(dn\log^2 n)$ flops are sufficient to
 compute 
a displacement generator of 
a length at most $d+2$
for the matrix $PMN$
 of any other
of these classes where
 $P$ and $M$ are from the set of 
 canonical multipliers complemented by the
identity matrix.
The flop bound decreases to
$O(dn\log n)$ where 
the canonical multipliers are DFT-based.
\end{corollary}

One can
 simplify the inversion 
of structured matrices $M$ of some important
classes
and 
the  solution of linear systems
$M{\bf x}={\bf u}$
by employing
 preprocessings
 $M\rightarrow PMN$
with appropriate
structured
 multipliers $P$ and $N$.
For an example 
we can decrease the complexity bound 
of $O(d^2n\log^3 n)$ 
of Corollary \ref{coinv} to
$O(d^2n\log^2 n)$
by applying 
canonical transformations
(\ref{eqprepr}) of
the matrices of
the  bottleneck  classes 
$\mathcal V$, $\mathcal V^T$ and
$\mathcal C$  
into the "easier" matrices
of the classes $\mathcal T$ and
$\mathcal H$.



\subsection{The impact on displacements}\label{strmstr}


 
In the canonical maps
of Theorem \ref{thmap} 
the displacement ranks grow
by at most 2 but possibly less
than that, as this is implied by 
Theorems \ref{thhv} and \ref{thdiag}
below.
In
our {\em constructive proofs} 
of these theorems
we also
 specify the  multipliers $P$ and $N$ and
compute the displacement generators  for the products
$PMN$ of some maps of Theorem \ref{thmap}.
In some maps we set   $P=I$ or $N=I$,
thus omitting one of the multipliers.
We show no maps where the matrices $M$ or $PMN$
belong to th classes $\mathcal V^T$, $\mathcal V^{-1}$, or 
$\mathcal V^{-T}$, but they can be
generated from the maps with 
$M\in \mathcal V$ or 
$PMN \in \mathcal V$
by means of transposition and inversion.

\begin{theorem}\label{thhv}
Given a displacement generator
of a length $d$ for a structured matrix $M$ 
of any of the four classes 
$\mathcal T$, 
$\mathcal H$,
$\mathcal V$, 
and
$\mathcal C$,
one can obtain a displacement generator of a length at most $d+2$
for 
a matrix $PMN$ belonging to any other of these classes by selecting
appropriate canonical multipliers 
$P$ and $N$
among the matrices $I$ (from the class $\mathcal T$), $J$
 (from the class $\mathcal H$), 
 Vandermonde  matrices $V$ and their transposes 
$V^T$.
Namely, if we assume canonical multipliers $P$ and $N$,
then we can compute a
displacement generator of the matrix $PMN$
having a length 
at most $d$ where the map $M\rightarrow PMN$
is between the matrices $M$ and $PMN$ in the classes
$\mathcal H$ and $\mathcal T$. This length
bound grows to
at most $d+2$ where            
$M$ is in the class $\mathcal T$ or $\mathcal H$,
whereas 
 $PMN \in\mathcal C$ or vice versa, 
where $M \in\mathcal C$
and $PMN$ 
is in the class $\mathcal T$ or $\mathcal H$.
We yield dispacement generators of at most lengths
  $d+1$ in the maps
$M\rightarrow PMN$ that support 
all other transitions among the classes
$\mathcal T$, $\mathcal H$,
$\mathcal V$ and $\mathcal C$.               
\end{theorem}
\begin{proof}
We specify some maps $M\rightarrow MNP$ 
that suport the claims of the theorem. 
One can 
vary and combine these maps
as well as the other maps of Theorem \ref{thmap}  and Table \ref{tabmsmp}. 

(a) $\mathcal T\rightarrow  \mathcal H$, $PMN=JM$.
Assume
a matrix $M\in \mathcal T$,
a pair of distinct 
scalars $e$ and $f$,
and a pair of $n\times d$ matrices $F=F_{Z_e,Z_f}(M)$ 
and $G=G_{Z_e,Z_f}(M)$ for
$d=d_{Z_e,Z_f}(M)$
satisfying the
displacement equation $Z_eM-MZ_{f}=FG^T$. 
Pre-multiply this  equation by the matrix $J$ to obtain
 $JZ_eM-(JM)Z_{f}=JFG^T$. 
 Rewrite the term
$JZ_eM=JZ_eJJM$ as $Z_e^TJM$
by observing that
$JZ_eJ=Z_e^T$ (cf.
(\ref{eqrever}) for $f=e$).
Obtain
$Z_e^T(JM)-(JM)Z_{f}=JFG^T$.
Consequently
$F_{Z_e^T,Z_f}(JM)=JF$,
$G_{Z_e^T,Z_f}(JM)=G$,
 $d_{Z_e^T,Z_{f}}(JM)=d_{Z_e,Z_{f}}(M)$, 
and  $JM\in \mathcal H$.

(b) $\mathcal T\rightarrow  \mathcal V$, $PMN=VM$.
Keep the assumptions 
 of part (a) and
fix 
 $n$ scalars  $s_1\dots,s_n$.
Pre-multiply the displacement equation 
$Z_eM-MZ_{f}=FG^T$
by 
the Vandermonde matrix $V=(s_i^{j-1})_{i,j=1}^n$
to
obtain $VZ_eM-(VM)Z_{f}=VFG^T$.
Write
${\bf s}=(s_i)_{i=1}^n$
and
substitute 
equation 
(\ref{eqvand})
 to yield
$D_{\bf s}(VM) -(VM)Z_{f}=VFG^T+(s_i^n-e)_{i=1}^n{\bf e}_n^TM=
F_{VM}G_{VM}^T$ for $F_{VM}=(VF~|~(s_i^n-e)_{i=1}^n)$ and
$G_{VM}^T=\begin{pmatrix}
G^T \\
{\bf e}_n^TM
\end{pmatrix}$. 
So 
 $ d_{D_{\bf s},Z_{f}}(VM) \le d_{Z_e,Z_{f}}(M)+1$
and $VM\in \mathcal V$. 

(c) $\mathcal H\rightarrow  \mathcal T$,
$PMN=MJ$.
Assume
a matrix $M\in \mathcal H$,
a pair of  
scalars $e$ and $f$,
and a pair of $n\times d$ matrices $F$ and $G$ for
$d=d_{Z_e,Z_f^T}(M)$
satisfying
the displacement equation $Z_eM-MZ_{f}^T=FG^T$.
Post-multiply it by the matrix $J$ to obtain
 $Z_e(MJ)-MZ_{f}^TJ=FG^TJ$.
Express the term
$MZ_f^TJ=MJJZ_f^TJ$ as $MJZ_f$ 
(cf. (\ref{eqrever}))
to obtain 
$Z_e(MJ)-(MJ)Z_{f}=FG^TJ=F(JG)^T$
and consequently
$F_{Z_e,Z_f}(JM)=F$,
$G_{Z_e,Z_f}(JM)=JG$,
 $d_{Z_e,Z_{f}}(MJ)=d_{Z_e,Z_{f}^T}(M)$ 
and $MJ\in \mathcal T$.

(d) $\mathcal H\rightarrow  \mathcal V$.
Compose the maps of 
parts (c) and (b).

(e) $\mathcal V\rightarrow  \mathcal H$,
 $PMN=V^TM$.
Assume
$n+2$  scalars
$e,f,
s_1,\dots,s_n$,
a matrix $M\in \mathcal V$, 
and its displacement generator
given by
$n\times d$ matrices $F$ and $G$ such that
$D_{\bf s}M-MZ_f=FG^T$.
Pre-multiply this equation by the  transposed Vandermonde 
matrix $V^T=(s_j^{i-1})_{i,j=1}^n$ to obtain
$V^TD_{\bf s}M-(V^TM)Z_f=V^TFG^T$
for
${\bf s}=(s_i)_{i=1}^n$.
Apply equation (\ref{eqvandt})
to express  the matrix $V^TD_{\bf s}$
 and obtain
$Z_e^T(V^TM)-(V^TM)Z_f=V^TFG^T-{\bf e}_n((s_i^n-e)_{i=1}^n)^TM=
F_{V^TM}G_{V^TM}^T$
for $F_{V^TM}=(V^TF~|~{\bf e}_n)$
and $G_{V^TM}^T=\begin{pmatrix}
G^T \\
((e-s_i^n)_{i=1}^n)^TM
\end{pmatrix}$.
So 
 $d_{Z_e^T,Z_{f}}(V^TM)\le d_{D_{\bf s},Z_f}(M)+1$
and  
$V^TM\in \mathcal H$. 

(f) $\mathcal V\rightarrow  \mathcal T$.
Compose  the maps of 
parts (e) and (c).

(g) $\mathcal V\rightarrow  \mathcal C$,
$PMN=MJV^T$.
Assume  
$2n+1$  scalars $e$,
$s_1,\dots,s_n,t_1,\dots,t_n$,
a matrix $M\in \mathcal V$, 
and its displacement generator
given by
$n\times d$ matrices $F$ and $G$.
Post-multiply the equation 
$D_{\bf s}M-MZ_e=FG^T$
by the matrix $JV^T$
where  $V^T=(t_j^{i-1})_{i,j=1}^n$
is the transposed Vandermonde 
matrix, substutute $Z_eJ=JZ_e^T$, and obtain
$D_{\bf s}(MJV^T)-MJZ_e^TV^T=FG^TJV^T$
for
 ${\bf s}=(s_i)_{i=1}^n$.
Apply equation (\ref{eqvandt}) to
express the matrix $Z_e^TV^T$ and
 obtain
$D_{\bf s}(MJV^T)-(MJV^T)D_{\bf t}=FG^TJV^T-MJ{\bf e}_n((t_i^n-e)_{i=1}^n)^T=
F_{MJV^T}G_{MJV^T}^T$
where $F_{MJV^T}=(F~|~MJ{\bf e}_n)$
and $G_{MJV^T}^T=\begin{pmatrix}
G^TJV^T \\
((e-t_i^n)_{i=1}^n)^T
\end{pmatrix}.$
So 
 $d_{D_{\bf s},D_{\bf t}}(MJV^T)\le d_{D_{\bf s},Z_e}(M)+1$
and  
$MJV^T\in \mathcal C$. 

We can alternatively write $PMN=MV^{-1}$
for $V=V_{\bf t}$. (The matrix $V$ can be readily inverted 
where it is FFT-based, that is where $V=V_f$.)
Post-multiply the equation $D_{\bf s}M-MZ_e=FG^T$
  by the matrix $V^{-1}=V_{\bf t}^{-1}$,
for ${\bf t}=(t_i)_{i=1}^n$,
to obtain $D_{\bf s}MV^{-1}-MZ_eV^{-1}=FG^TV^{-1}$.
Pre- and post-multiply
by $V^{-1}$
equation
(\ref{eqvand})  for ${\bf s}$ replaced by  ${\bf t}$ and
obtain $Z_eV^{-1}=V^{-1}D_{\bf t}-V^{-1}(t_i^n-e)_{i=1}^n{\bf e}_n^TV^{-1}$.
Substitute the expression of $Z_eV^{-1}$ from this equation
into above equation
 and obtain
$D_{\bf s}(MV^{-1})-(MV^{-1})D_{\bf t}=FG^TV^{-1}-V^{-1}(t_i^n-e)_{i=1}^n{\bf e}_n^TV^{-1}=F_{MV^{-1}}G_{MV^{-1}}^T$
for $F_{MV^{-1}}=(F~|~V^{-1}(t_i^n-e)_{i=1}^n)$
and $G_{V^{-1}M}^T=\begin{pmatrix}
G^TV^{-1} \\
{\bf e}_n^TV^{-1}
\end{pmatrix}$.
So 
 $d_{D_{\bf s},D_{\bf t}}(VM)\le d_{D_{\bf s},Z_e}(M)+1$
and  
$MV^{-1}\in \mathcal C$.

(h) $\mathcal C\rightarrow  \mathcal V$,
$PMN=MV$.
Assume 
$2n+1$ scalars $e$,
$s_1,\dots,s_n,t_1,\dots,t_n$,
a matrix $M\in \mathcal C$, 
and its displacement generator
given by
$n\times d$ matrices $F$ and $G$ such that
$D_{\bf s}M-MD_{\bf t}=FG^T$
for ${\bf s}=(s_i)_{i=1}^n$
and ${\bf t}=(t_i)_{i=1}^n$.
Post-multiply this equation by the  Vandermonde 
matrix $V=(t_i^{j-1})_{i,j=1}^n$ to obtain
$D_{\bf s}(MV)-MD_{\bf t}V=FG^TV$.
Express the matrix $D_{\bf t}V$
from matrix equation (\ref{eqvand})
 and obtain
$D_{\bf s}(MV)-(MV)Z_e=FG^TV+M(t_i^n-e)_{i=1}^n{\bf e}_n^T=
F_{MV}G_{MV}^T$
where $F_{MV}=(F~|~M(t_i^n-e)_{i=1}^n)$
and $G_{MV}^T=\begin{pmatrix}
G^T V\\
{\bf e}_n^T
\end{pmatrix}.$
So 
 $d_{D_{\bf s},Z_e}(MV)\le d_{D_{\bf s},D_{\bf t}}(M)+1$
and  
$MV\in \mathcal C$. 

(i) $\mathcal C\rightarrow  \mathcal T$.
Compose  the maps of parts (h) and (f).

(j) $\mathcal C\rightarrow  \mathcal H$.
Compose  the maps of parts (h) and (e).

(k) $\mathcal T\rightarrow  \mathcal C$.
Compose  the maps of parts (b) and (g).

(i) $\mathcal H\rightarrow  \mathcal C$.
Compose  the maps of parts (d) and (g).
\end{proof}                   


Multiplications by a Cauchy matrix
keeps a matrix in any of the classes
$\mathcal T$, $\mathcal H$, $\mathcal V$
and $\mathcal C$, but changes 
a diagonal operator matrix. 
Next we specify  
the impact on the displacement.

\begin{theorem}\label{thdiag}
Assume $2n$ distinct scalars 
$s_1,\dots,s_n,t_1,\dots,t_n$,
defining two vectors
${\bf s}=(s_i)_{i=1}^n$ and ${\bf t}=(t_j)_{j=1}^n$ and
a nonsingular Cauchy
matrix $C=C_{\bf s, t}=(\frac{1}{s_i-t_j})_{i,j=1}^n$
(cf. part (i) of Theorem \ref{thcv}).
Then for any pair of operator matrices $A$ and $B$ we have

(i) $d_{A,D_{\bf t}}(MC)\le d_{A,D_{\bf s}}(M)+ 1$ and
  
(ii) $d_{D_{\bf s},B}(CM)\le d_{D_{\bf t},B}(M)+ 1$.
\end{theorem}
\begin{proof}        
(i) We have $d_{A,D_{\bf s}}(M)=\rank (AM-MD_{\bf s})
=\rank(AMC-MD_{\bf s}C)$.
Furthermore
$AMC-MCD_{\bf t}=AMC-MD_{\bf s}C+MD_{\bf s}C-MCD_{\bf t}
=(AM-MD_{\bf s})C+M(D_{\bf s}C-CD_{\bf t})$.
Substitute   equation (\ref{eqcch}) and deduce that 
$AMC-MCD_{\bf t}=(AM-MD_{\bf s})C+M{\bf e}{\bf e}^T$.
Therefore $d_{A,D_{\bf t}}(MC)=\rank (AMC-MCD_{\bf t})\le 
\rank ((AM-MD_{\bf s})C)+1=\rank (AM-MD_{\bf s})+1
=d_{A,D_{\bf s}}(M)+1$.

(ii) We have 
$D_{\bf s}CM-CMB=D_{\bf s}CM-CD_{\bf t}M+CD_{\bf t}M-CMB=
(D_{\bf s}C-CD_{\bf t})M+
C(D_{\bf t}M-MB)$.
Substitute   equation (\ref{eqcch}) and deduce that
$D_{\bf s}CM-CMB=C(D_{\bf t}M-MB)+{\bf e}{\bf e}^TM$.
Therefore $d_{D_{\bf s},B}(CM)= \rank (D_{\bf s}CM-CMB)\le
 \rank (C(D_{\bf t}M-MB))+ 1=
\rank (D_{\bf t}M-MB)+ 1=
d_{D_{\bf t},B}(M)+1$.
\end{proof}                 


\subsection{Canonical and DFT-based transformations of the matrices of the classes $\mathcal T$,  $\mathcal H$,
$\mathcal V$ and $\mathcal V^T$
into CV-like  matrices}\label{scv}


Multiplication by a Vandermonde multiplier 
$V=(s_i^{j-1})_{i,j=1}^n$ or
by its transpose in 
 Corollary \ref{codgdr}
increases 
the length of a displacement generator 
by at most $1$, but
in the proof of Theorem \ref{thhv}
such a multiplication
 does not increase the length at all
where $s_i^n=e$ for $i=1,\dots,n$ 
and for a scalar $e$, employed 
in the operator matrices $Z_e$ and $Z_e^T$
of 
the Van\-der\-monde displacement map
(cf. (\ref{eqvand}) and (\ref{eqvandt})).
This suggests choosing the vectors 
${\bf s}=(e\omega_n^{i-1})_{i=1}^n$
and ${\bf t}=(f\omega_n^{i-1})_{i=1}^n$
and employing the DFT-based multipliers $V_e$
and $V_f$
(cf. (\ref{eqvf}))
wherever we are free to choose these vectors
and multipliers.
In particular we can choose such DFT-based multipliers 
in our maps
supporting part (g)  of Theorem \ref{thhv}, 
and then we would output
 matrices of the class $\mathcal {CV}$
having the same displacement ranks as the 
input matrices $M$.
Furthermore the inverse of the matrix  $V=V_{\bf t}$,
 employed in our second map supporting part (g), 
would turn into DFT-based matrix $V_f$,
and we could invert it
and multiply it by a vector
by using $O(n\log n)$ flops
(cf. Theorem \ref{thcvdft}).
We deduce the following results by reexamining the proof of 
Theorem \ref{thhv} and applying transposition.

\begin{theorem}\label{thfmap}
Some appropriate  canonical DFT-based
multipliers 
from the proof of
 Theorem \ref{thhv}
for the basic vectors 
${\bf s}=(e\omega_n^{i-1})_{i=1}^n$
and ${\bf t}=(f\omega_n^{i-1})_{i=1}^n$
 support the following transformations of matrix classes
(in both directions),
 $\mathcal T\leftrightarrow \mathcal {FV}\leftrightarrow \mathcal {FCF}$,
 $\mathcal H\leftrightarrow \mathcal {FV}\leftrightarrow \mathcal {FCF}$,
 $\mathcal V\leftrightarrow \mathcal {CF}$, 
 $\mathcal V^T\leftrightarrow \mathcal {FC}$, and
$\mathcal V\cup \mathcal V^T\leftrightarrow \mathcal {CV}$.
The multipliers
are quasiunitary where $|e|=|f|=1$. 
\end{theorem}

By combining our second map of the proof of part (g)
of  Theorem \ref{thhv} with our map from its part (b)
and choosing 
${\bf t}=(f\omega_n^{i-1})_{i=1}^n$,
we can obtain  canonical DFT-based transforms 
$\mathcal T\rightarrow  \mathcal C=\Omega\mathcal T\diag(f^{i-1})_{i=1}^n\Omega^H$, 
which are quasiunitary where $|f|=1$.
For $f=\omega_{2n}$ they turn into the celebrated map
employed in the papers \cite{H95},
\cite{GKO95},
\cite{G98},
\cite{MRT05},
\cite{R06},
\cite{CGS07},
\cite{XXG12}.
The following theorem shows the implied map of
the displacement generators.

\begin{theorem}\label{thtc}
Suppose $Z_1M-MZ_{-1}=FG^T$
for an $n\times n$ matrix $M$ and 
$n\times d$ matrices $F$  and $G$
and write $P=\Omega_n$, $N=\omega_{2n}^{-1}\Omega_n^H$,
$C=PMN$, $D=\diag(\omega_{2n}^{i-1})_{i=1}^n$, and $D=D_0^2=\diag(\omega_{n}^{i-1})_{i=1}^n$. Then 
$DC-\omega_{2n}CD=F_CG_C^T$ for
$F_C=\Omega_nF$ and $G_C=\omega_{2n}\Omega_nD_0G$.
\end{theorem}

This theorem and the supporting 
 canonical DFT-based map 
$\mathcal T\rightarrow \mathcal C$
are the special cases of Theorem \ref{thhv} 
and its transforms of matrix structures,
 but they
appeared in \cite{H95}
as corollaries of Theorem \ref{thcpw}.
In his letter of 1991,
 reproduced in \cite[Appendix C]{P11},  G. Heinig 
acknowledged studying the  paper \cite{P90}, but
his alternative derivation in 
 \cite{H95} appeared ad hoc
and has defined a more narrow class of transforms of 
matrix structures than 
 Theorem \ref{thhv}, extending \cite{P90}.
Heinig's specialization of the structure transformation method, 
however, has paved way to the subsequent strong demonstration of the power 
of the method in \cite{GKO95},
\cite{G98},
\cite{MRT05},
\cite{R06},
\cite{CGS07},
\cite{XXG12}
and has specified
an efficient quasiunitary map  $\mathcal T\rightarrow \mathcal C$
 above, which employed
the uniform distribution of the $2n$ knots $s_1,\dots,s_n,t_1,\dots,t_n$
on the unit circle $\{z:~|z|=1\}$.


\section{HSS matrices}\label{shss}


The following class of  structured
matrices extends the class of banded matrices and their inverses.

\begin{definition}\label{defqs}
Hereafter ``{\em HSS}" stands for ``hierarchically semiseparable".
An $n\times n$ matrix is $(l,u)$-HSS 
if its diagonal blocks consist of $O((l+u)n)$
entries, if $l$  is the maximum rank of all its subdiagonal blocks, 
and if $u$  is the maximum rank of all its superdiagonal blocks,
that is blocks of all sizes lying strictly below or strictly above the block 
 diagonal, respectively.
\end{definition}

This definition is
one of a number of similar definitions of such
  matrices, also known 
under the names
of quasi\-se\-parable, weakly,
recursively or 
sequentially semiseparable matrices, as well as matrices with low Hankel rank
and rank structured matrices. See \cite{VVGM05}, 
\cite{VVM07}, \cite{VVM08}, 
and the bibliography therein on the long
history of the study of these matrix classes
and see \cite{GR87}, \cite{LRT79}, \cite{PR93}
 on the related subjects of 
Multipole and Nested Dissection  algorithms.

A banded matrix $B$
having a
lower bandwidth $l$ and an upper bandwidth $u$ 
is an $(l,u)$-HSS matrix, and so is its inverse $B^{-1}$
if the matrix $B$ is nonsingular. 
It is well known that such a banded $n\times n$  matrix
can be multiplied by a vector by using
 $O((l+u)n)$ flops, 
whereas   $O((l+u)^2n)$ flops
are sufficient to solve
a nonsingular linear system of $n$
equations with such a coefficient matrix.
Both properties
have been extended to $(l,u)$-HSS matrices
of the size $n\times n$ 
 (see \cite{MRT05}, \cite{CGS07}, \cite{XXG12}).
Furthermore, like the matrices of the classes
$\mathcal T$, $\mathcal H$, $\mathcal V$
and $\mathcal C$, 
such an HSS matrix allows 
its compressed representation
where
one  defines its generalized
generator that readily expresses
its $n^2$ entries  via $O((l+u)n)$ parameters.
The inverse of a nonsingular  $(l,u)$-HSS
$n\times n$ matrix $M$ is also an $(l,u)$-HSS
$n\times n$ matrix, and 
 a generator expressing the inverse
via $O((l+u)n)$ parameters can be computed
by using $O((l+u)^2n)$ flops.
See \cite{XXG12} and references therein
on 
the supporting algorithms and their
efficient implementation.


\section{Numerical ranks of Cauchy and Van\-der\-monde matrices}\label{snrqs}


Next we 
bound
numerical rank
for
a large subclass of the class 
of 
 Cauchy and
Cau\-chy-like
 matrices. In the next section 
we
extend this study to
approximate CV and CV-like matrices  
 by HSS matrices.

\begin{definition}\label{defss}
A pair of complex points $s$ and $t$
is $(\theta,c)$-{\em separated} 
for $\theta<1$
and a complex point $c$
if $|\frac{t-c}{s-c}|\le \theta$.
Two sets  of complex numbers  $\mathbb S$ and 
$\mathbb T$ are $(\theta,c)$-{\em separated} from one another
if every pair of elements $s\in \mathbb S$
and $t\in \mathbb T$ 
is $(\theta,c)$-separated from one another
for the same pair  $(\theta,c)$.
$\delta_{c,\mathbb S}=\min_{s\in \mathbb S}|s-c|$ and 
$\delta_{c,\mathbb T}=\min_{t\in \mathbb T}|t-c|$
denote the distances of the center $c$ from the sets $\mathbb S$
and $\mathbb T$, respectively.
\end{definition}

\begin{lemma}\label{less} \cite{R85}.
Suppose two complex points $s$ and $t$ are
 $(\theta,c)$-separated from one another
for  $0\le \theta<1$ and write $q=\frac{t-c}{s-c}$,
$|q|\le \theta$. Then for every 
positive integer $k$ we have


\begin{equation}\label{eqthetac}
\frac{1}{s-t}=
\frac{1}{s-c}\sum_{i=0}^{k-1}\frac{(t-c)^k}{(s-c)^k}+\frac{q_k}{s-c}~
{\rm where}~|q_k|\le \theta^k/(1-\theta).
\end{equation}
\end{lemma}
\begin{proof}
$\frac{1}{s-t}=\frac{1}{s-c}~~\frac{1}{1-q}=
\frac{1}{s-c}\sum_{i=0}^{\infty}q^i=
\frac{1}{s-c}(\sum_{i=0}^{k}q^i+\sum_{i=k}^{\infty}q^i)=
\frac{1}{s-c}(\sum_{i=0}^{k}q^i+\frac{q^k}{1-q})$.
\end{proof}

\begin{corollary}\label{coss}  (Cf. \cite{MRT05}, \cite[Section 2.2]{CGS07}.)
Suppose  $C=(\frac{1}{s_i-t_j})_{i,j=1}^{n}$
is a Cauchy matrix defined by two sets 
of parameters
$\mathbb S=\{s_1,\dots,s_n\}$ and 
$\mathbb T=\{t_1,\dots,t_n\}$. 
 Suppose these sets are $(\theta,c)$-{\em separated} from one another
for $0<\theta<1$ and a scalar $c$
and write 
\begin{equation}\label{eqdlt}
\delta=\delta_{c,\mathbb S}=\min_{i=1}^{n} |s_i-c|.
\end{equation}
 Then for every positive integer $k$ 
it is sufficient to use $2kn+4n$ ops to
 compute
 two matrices 
\begin{equation}\label{eqgh}
F=(1/(s_i-c)^h)_{i,h=1}^{n,k+1},~G^T=((t_j-c)^h)_{j,h=0}^{n,k}
\end{equation}
that support the representation of the matrix 
$C$ as $C=\widehat C+E$ where 
\begin{equation}\label{eqrnk} 
\widehat C=FG^T,~\rank (\widehat C)\le k+1,
\end{equation}
\begin{equation}\label{eqappr} 
E=(e_{i,j})_{i,j=1}^{n},~|e_{i,j}|\le \frac{q^k}{(1-q)\delta}~{\rm for~all~pairs}~ \{i,j\},
\end{equation}
and so 
$||E||\le nq^k/((1-q)\delta)$.
\end{corollary}
\begin{proof}
Apply (\ref{eqthetac}) for $s=s_i$, $t=t_j$
and all pairs $\{i,j\}$ to deduce (\ref{eqappr}). 
\end{proof}

Here are three immediate extensions of the 
theorem and the
corollary.

(i) We can replace $\delta=\delta_{c,\mathbb S}=\min_{i=1}^{n} |s_i-c|$
by $\delta=\delta_{c,\mathbb T}=\min_{j=1}^{n} |t_j-c|$
because 
$C_{\bf s,t}^T=-C_{\bf t,s}$ (cf. (\ref{eqctr})).

(ii) By virtue of part (c) of Theorem \ref{thdexpr}
we
can extend 
the bounds 
of 
Corollary \ref{coss}
from a Cauchy matrix 
$C_{\bf s,t}$ 
to a Cauchy-like  matrix
of the class 
$\mathcal C_{\bf s,t}$, 
 given with a displacement generator $(F,G)$
of a  length $d$. In this extension  
the rank bound (\ref{eqrnk}) 
increases
by  a factor of $d$ and 
the error norm bound (\ref{eqappr})
increases by  a factor of $d~||F||~||G||$.

(iii) Already for moderately large integers $k$
the upper bounds of (\ref{eqappr}) are small 
unless the values $1-\theta>0$ and  
$\delta$ of (\ref{eqdlt})
are small. Then 
Corollary  \ref{coss}  
implies an upper bound $k+1$ on the numerical rank of 
the large subclass of
 Cauchy  matrices $C=(\frac{1}{s_i-t_j})_{i,j=1}^{n}$
whose
parameter sets $\mathbb S=\{s_1,\dots,s_n\}$ and 
$\mathbb T=\{t_1,\dots,t_n\}$ 
are $(\theta,c)$-separated from one another
for an appropriate center $c$.
If this property 
holds for 
two  subsets of the sets $\mathbb S$
and $\mathbb T$ that define an $n\times l$ or an $l \times n$ Cauchy submatrix
where $l>k+1$, then 
the
$l$ rows or columns of this submatrix form a 
nearly rank deficient matrix,
which means that  the matrix $C$ is ill conditioned.
Apply a canonical DFT-based
quasiunitary
 map 
$\mathcal V\rightarrow \mathcal {CV}$
that supports
part (g)
of Theorem \ref{thhv} 
(see Theorem \ref{thfmap}) and
deduce that 
a 
Van\-der\-monde  matrix
$V_{\bf t}$
is ill conditioned
 unless its knots 
from the set 
$\mathbb T=\{t_1,\dots,t_n\}$
are close enough to all or almost all knots 
of the set 
$\{\omega_n^{i-1}\}_{i=1}^n$
of the $n$th 
roots of 1, scaled by  a scalar $e$, $|e|=1$.
This implies (cf.
 \cite{GI88})
that except for a narrow subclass
all Van\-der\-monde matrices 
are ill conditioned.


\section{HSS approximation of CV and CV-like  
matrices}\label{shssapr}


\begin{theorem}\label{thlrba}
Assume positive integers $g$, $h$ and $n$, 
a scalar $e$,  
and a Cauchy matrix $C=C_{{\bf s},e}=(\frac{1}{s_i-t_j})_{i,j=1}^{n}$ 
such that 
$t_j=e(\omega_n^{j-1})$ for $j=1,\dots,n$ (cf. Section \ref{sfour}),
$gh=n$, $n$ is not small, and $|e|=1$. 
Then there is a 
permutation $n\times n$  matrix $P$ such that
 $CP$ is a $3\times g$ block matrix
with block columns 
$(C_{j,-}^T~|~\Sigma_j^T~|~C_{j,+}^T)^T$, $j=0,\dots,g-1$,
where
the diagonal blocks $\Sigma_j$ have
 sizes $n_j \times h$,
and 
the rows of the blocks
$\Sigma_j$ and $\Sigma_k$ lie in pairwise distinct sets of rows of the matrix $CP$
unless $|j-k|\le 1$ or 
$|j-k|=g-1$
 (and so the blocks $\Sigma_1,\dots,\Sigma_g$
together have at most $3hn$ entries), whereas every matrix $(C_{j,-}^T~|~C_{j,+}^T)^T$
is an $h\times (n-n_j)$ Cauchy matrix defined by the sets of parameters 
that are $(1/2,c_j)$-separated from one another
for some scalars $c_j$ lying 
on the unit circle $\{z:~|z|=1\}$ and
at the distance of at least $0.5h/n^2$ from the set $\mathbb S_j$.
\end{theorem}

\begin{proof} 
Represent the knots $s_1,\dots,s_n$ of the set $\mathbb S$ in polar coordinates,
$s_i=r_i\exp(2\pi\phi_i\sqrt {-1})$ where 
$r_i\ge 0$, $0\le\phi_i<2\pi$, $\phi_i=0$ if $r_i=0$,
and $i=0,1,\dots,n-1$. Re-enumerate all values $\phi_i$
to have them in nonincreasing order and
 to have $\phi_0^{(\rm new)}=\min_{i=0}^n\phi_i$
and let $P$ denote the permutation matrix that defines this re-enumeration.
To simplify our notation assume 
 that already 
the original enumeration has these  properties
and that $e=1$.
Let $\mathbb S_j=\{s_j\}_j\in \mathbb S$ and $\mathbb T_j=\{\omega_n^{l}\}_{l=jh}^{j(h+1)-1}\in \mathbb T$ 
 denote the sets of knots 
lying in the semi-open sectors of the complex plane 
bounded by the pairs of rays 
 from
 the origin to
the points $\omega_n^{jh}$ 
and $\omega_n^{(j+1)h}$,
respectively. Namely denote by 
 $\mathbb S_j$ and  $\mathbb T_j$
 the subsets 
of the sets $\mathbb S$ and  $\mathbb T$
made up
of the
knots 
whose arguments 
$\phi_j$ satisfy                                                                                                                           $2\pi jh/n                                                                                                                             \le \phi_j<2\pi (j(h+1)-1)/n$,
$j=0,\dots,g-1$.

Write $\alpha(a,b)$ to denote 
the arc of the unit circle $\{z:~|z|=1\}$
with the end points $a$ and $b$.
For every $j$, $j=1,\dots,g$,
choose a center $c_j$
on the arc $\alpha(\omega_{4n}^ {(4j+1)h},\omega_{4n}^{(4j+3)h})$.
This arc has  the length $\pi h/n$ 
and shares the midpoint $\omega_{2n}^{(2j+1)h}$
with
 the arc 
$\alpha(\omega_n^{jh},\omega_n^{(j+1)h})$,
having the length $2\pi h/n$. Choose the center $c_j$ at the distance at least
$2h/n^2$
from the set $\mathbb S$
(as we required). This is possible because 
the set has exactly $n$ elements.
For $j=0,\dots,g-1$,
index by $jh,\dots,j(h+1)-1$ the columns  shared by
 the blocks $C_{j,-}$, $\Sigma_j$ and $C_{j,+}$
  and
index  the rows of the blocks $\Sigma_j$
by the indices of the elements of the set $\mathbb S_{j-1}\cup\mathbb S_j\cup\mathbb S_{j+1}$.
Note that the sets $\mathbb S_j$ and $\mathbb T_k=\{\omega_n^{l}\}_{(k-1)h}^{kh-1}$
are $(1/2,c_j)$-separated from one another unless $|j-k|\le 1$ or $|j-k|=g-1$,
and this implies the separation property
claimed in
 the theorem.
\end{proof}

Apply Corollary \ref{coss} for $q=1/2$, $\delta=0.5h/n^2$,
$C=(C_{u,-}~|~C_{u,+})^T$, and $u=1,\dots,g$
and obtain the following corollary.

\begin{corollary}\label{colrbap}
The matrix $PC$ of Theorem \ref{thlrba}
can be represented as 
\begin{equation}\label{eqlrbap} 
PC=\Sigma+\widehat C+E
\end{equation}
where $\Sigma$
is the block diagonal matrix
$\diag(\Sigma_u)_{u=1}^g$, 
$\rank (\widehat C)\le (k+1)g$, 
$E=(e_{i,j})_{i,j=1}^{n},~|e_{i,j}|\le n^2 2^{2-k}/h$ for all pairs  $\{i,j\}$,
and so
$||E||\le n^3 2^{2-k}/h$.
\end{corollary}

\begin{remark}\label{reexts}
Theorem \ref{thlrba} and the corollary can be immediately extended 
to the case where $h$ does not divide $n$ (in this case write $g=\lceil n/h \rceil$) 
as well as  to the case where $C=(\frac{1}{s_i-t_j})_{i,j=1}^{n}$ for 
$s_i=e\omega^{i-1}$ for all $i$ and $|e|=1$
(because 
$C_{{\bf s},e}=-C_{e,{\bf s}}^T$ (cf. (\ref{eqctr}))).
Theorem \ref{thdexpr} implies an extension to the matrices $M$
of the class $\mathcal CV$, 
with the increase of the rank bound 
by a factor of $d$ and with
the increase of the approximation norm bound 
by  a factor of $d~||F||~||G||$ 
provided the matrix $M$ is given with its  displacement
generator $(F,G)$ of a length $d$.
The proof technique of Theorem \ref{thlrba}  
enables various further extensions.
Clearly one can allow any variation of the set $\mathbb T$
as long as its elements can be partitioned 
into $h$-tuples, each lying on 
or near 
the arc of  the unit circle $\{z:~|z|=1\}$
with the endpoints $\omega_n^{jh}$ 
and $\omega_n^{(j+1)h}$.
Furthermore the proof can be readily
extended 
  to the case where a 
 line  interval of a length between 1 and 2 (say) lying
on the complex plane  not very far from the origin 
(or on an  approximation 
of such a line interval
by a segment of a curve)
replaces the unit circle
$\{z:|z|=1\}$ and where 
the set $\mathbb T$ can be partitioned into
$h$-tuples that  are more or less 
equally spaced on this interval
(or the segment).
\end{remark}


The block diagonal matrix $\Sigma$ has at most $3hn$ entries.
The matrix $\widehat C$ consists of the off-diagonal blocks.
By combining 
Theorem \ref{thlrba} and
Corollary \ref{colrbap} with the HSS techniques
of \cite{GR87},  \cite{MRT05},  \cite{CGS07},  \cite{XXG12},
deduce that 
for a positive constant 
$b$ and  the integer $k=\lceil 3(b+2)\log_2n\rceil$,
the  matrix
$\widehat C$ 
of (\ref{eqlrbap})
is an $(l,u)$-HSS  matrix 
where $l+u\le ckh$, $h\le c'\log n$, 
 $n^3 2^{2-k}/h\le 2^{-b}$,
and  $c$ and $c'$ are
two constants.


\section{Multiplication of the
matrices  of the classes $\mathcal {CV}$,
$\mathcal {V}$, $\mathcal {V^T}$, and
  $\mathcal {C}$
and their inverses 
by vectors 
}\label{svmv}


Suppose  $\mu(M)$ denotes the minimum number of flops sufficient
for multiplying a matrix $M$
by  a vector and  estimate $\mu(C)=\mu(PC)$
for the matrices of Corollary \ref{colrbap}.
The matrix $\Sigma$
has at most $3hn$ nonzero entries,
and so $\mu(\Sigma)\le 6hn-n$.
Furthermore
$\mu(\widehat C)=O(n\log n)$
because the matrix $\widehat C$
has the $(l,u)$-HSS structure for 
$l+u\le ckh$ and $h\le c'\log n$
(see Section \ref{shss}).
Let us summarize the estimates for the 
CV matrices with an extension to the matrices 
of the classes $\mathcal {CV}$,
$\mathcal {V}$, and $\mathcal {V^T}$.

\begin{theorem}\label{thvv} (See Remark \ref{reimpl}.)
 Assume a positive scalar $b$, a complex $e$ such that $|e|=1$,
 and 
two vectors ${\bf f}$ and ${\bf s}$ 
of dimension $n$.
(i) Then one can approximate the product $M{\bf f}$
within the  error norm bound $2^{-b}~||M||~||{\bf f}||$
by using $O(bn\log n)$ flops
provided that $M$
is a  CV,  Van\-der\-monde  
or  transposed Van\-der\-monde $n\times n$ matrix
 $C_{{\bf s},e}$, $C_{e,{\bf s}}$,  $V_{\bf s}$ or $V_{\bf s}^T$, respectively.
(ii) The flop bound 
for solving a nonsingular linear system of $n$ 
equations with the coefficient matrix in the above classes
increases versus part (i) by 
a factor of $\log n$ 
and the error norm bounds 
increases by  
a factor of $||M^{-1}||/||M||$.
(iii) The flop bounds of parts (i) and (ii)
also hold for approximate  evaluation of a polynomial of degree
$n-1$ at $n$ points and for approximate interpolation to
this polynomial from its $n$ values, 
respectively.
(iv) The flop bounds of parts (i) and (ii) increase by 
a factor of
$d$,
whereas the error norm bounds 
increase by  
a factor of $d~||F||~||G||$
where $M$ is a matrix from the class
$\mathcal C_{{\bf s},e}$, $\mathcal C_{e,{\bf s}}$,  $\mathcal V_{\bf s}$ or $\mathcal V_{\bf s}^T$
 (having the structure of CV, Van\-der\-monde  
or  transposed Van\-der\-monde 
type) given with a displacement generator 
$(F,G)$
of a length 
$d$.
\end{theorem}
\begin{proof}
Summarize our estimates above to deduce the bound of part (i)
in the case of  CV matrices  $C_{{\bf s},e}$ and $C_{e,{\bf s}}$.
Apply Theorem \ref{thpert} to estimate the approximation 
errors of solving the linear systems of equations
and extend the bounds of part (i) to part  (ii).
To extend the estimates of parts (i) and (ii)
to the case of Van\-der\-monde matrices
 $V_{\bf s}$, apply the canonical DFT-based specialization of a 
map supporting part (g) of Theorem \ref{thhv} 
for the DTF-based matrix $V=V_f$ where
 $|f|=1$, and so  $||V_f||=\sqrt n$.
The map increases the approximation error norm 
(versus the case of  CV matrices  $C_{{\bf s},e}$ and $C_{e,{\bf s}}$)
by 
a factor of  $\sqrt n\min_{i=1}^n\frac{1}{|s_i^n-f|}$. 
Choose a complex $f$,
$|f|=1$, that 
keeps this factor below $3n \sqrt n$.
Compensate for this increase of the norm bound by
adding $\log_2(3n\sqrt n)$ to the value $k$.
Similarly  multiply a
transposed Van\-der\-monde matrix by a vector,
transpose a map that supports part (g) of Theorem \ref{thhv},
and employ equation (\ref{eqctr}).
Extend the results of parts (i) and (ii) to part (iii)
 by applying   Theorem \ref{thevint}. Extend them
 to
 part (iv) by applying  parts $(v)$,  $(v^T)$, and $(c)$
Theorem \ref{thdexpr},
choosing a scalar $e$ in parts $(v)$ and $(v^T)$
such that $\min_{i=1}^n|e-s_i^n|\ge 1/2$ (say),
and increasing the integer parameter $k$ by $\lceil \log (e n)\rceil$
(to compensate for the exceess of the norms $||Z_e({\bf f}_i)||$  
and  $||Z_e({\bf g}_i)||$ above $||F||$ and
$||G||$, respectively).
\end{proof}

\begin{remark}\label{recrd}
One can ignore the HSS structure of the matrix 
$\widehat C$ 
 and still approximate  
the matrix product 
$C_{{\bf s},e}{\bf f}$
at the cost bounds
that are smaller than the
known bounds by a factor of
$\sqrt {n/\log n}$. Indeed 
choose $h$ of about  $\sqrt{n\log n}$
and choose $g$ of about $\sqrt{n/\log n}$
in Corollary \ref{colrbap}
and obtain the matrix 
$\widehat C$ of a rank 
of order $\sqrt{n\log n}$.
We can multiply this matrix
by a vector  
by using $O(n\sqrt{n\log n})$ flops.
The estimate is extended to 
the
overall cost of multiplying the 
matrix $\Sigma+\widehat C$ by a vector
because we can multiply the 
matrix $\Sigma$ by a vector
by using $6hn-n$ flops and because $h=O(\sqrt{n\log n})$.
\end{remark}

\begin{remark}\label{reexts1}
We can  extend Theorem \ref{thvv} similarly to
the extensions  
 of Theorem \ref{thlrba} and Corollary \ref{colrbap}
in the second part of
Remark \ref{reexts}.
\end{remark}

\begin{remark}\label{reimpl} 
The algorithms supporting
Theorem \ref{thvv}
can be naturally partitioned into two stages.
At first we apply canonical DFT-based transformations of 
Theorems \ref{thhv}--\ref{thfmap}
and \ref{thvv}
and Corollary \ref{coss} to reduce our 
tasks to computations with HSS matrices. 
At this stage we propose a novel specialization of the approach of \cite{P90}.
Then it remains to apply the Multipole
algorithms, which is both powerful and well developed.
We perform the former (FFT-based) stage by applying $O(n\log n)$
flops. The latter (Multipole/HSS) stage involves $O((l+u)n)$
flops for multiplication of an $n\times n$
HSS matrix by a vector and $O((l+u)^2n)$
 flops for solving a nonsingular HSS linear system of $n$ equations,
and we have the bound $l+u=O(\log n)$ in our case.
Empirically, however, 
in the extensive tests in \cite{XXG12} for
 HSS computations similar to ours, 
the value $(l+u)^2$ grew much slower than
$\log n$ as $n$ grew large, and so
 we can expect that
the  computational cost at the first (FFT) stage
 of the
algorithms actually dominates 
their overall computational cost.
\end{remark}




Clearly, the algorithms supporting
Theorem \ref{thvv}
 are efficient not only for CV and CV-like  
matrices, but for a larger subclass of
the class of Cauchy-like matrices
 (cf.  Remark \ref{reexts}). The extension to
the general Cauchy and Cauchy-like matrices
can lead to numerical problems, however. 
Here are some sketchy comments.
Suppose  that
 ${\bf s}$,
${\bf t}$
and  ${\bf u}$ 
denote three  vectors
of dimension $n$ and that
an $n\times n$ Cauchy-like matrix 
$M\in \mathcal C_{\bf s,t}$ is
given with a displacement generator $(F,G)$
of a length $d$.
Then for a large class of vectors 
${\bf s}$
and  ${\bf t}$, one can 
extend Theorem \ref{thlrba}
and reduce the 
approximation of the vectors $M{\bf u}$
and of
the solution ${\bf x}$ 
to a linear system of $n$ equations $M{\bf x}={\bf u}$
(if it is  nonsingular) 
 to HSS computations (cf. Remark \ref{reexts}).

Furthermore for all input vectors 
 ${\bf s}$,
${\bf t}$
and  ${\bf u}$, 
we can
apply 
 our techniques of transforming matrix
structures to reduce  
 the solution ${\bf x}$ 
of the linear system $M{\bf x}={\bf u}$
to some
computations with CV matrices and 
to 
the computation 
of the product of the matrix $M$ by the 
vector ${\bf e}$
as follows.
Fix a scalar $e$, write
 $P=MC_{{\bf t},e}$ and ${\bf x}=C_{{\bf t},e}{\bf y}$, and note that
$P{\bf y}=${\bf u}, whereas
$P\in \mathcal C_{{\bf s},e}$ is a CV matrix
with the displacement generator $(F_P,G_P)$
of length at most $d+1$
where $F_P=(F~|~M{\bf e})$
and $G_P=(C_{{\bf t},e}^TG~|~{\bf e})$.
By applying these techniques to the matrix $M^T\in \mathcal C_{\bf t,s}$
we can alternatively reduce the linear system 
$M{\bf x}={\bf u}$
to the computation of the products 
$M^T{\bf e}$ and to some
computations with CV matrices.
In both cases application of the algorithms would require 
additional error analysis. E.g., the
approximation errors of computing
the matrix $P$ would magnify the 
approximation
errors
for the vectors ${\bf y}$
(cf. Theorem \ref{thpert})
and  ${\bf x}$.

Next we  
consider another
extension of our techniques
and make further comments on  error propagation.
 Part (c) of Theorem \ref{thdexpr}
enables us to reduce the
 approximation of  the vector
${\bf x}=M{\bf u}$
to the approximation of  the $d$ vectors
$C_{\bf s,t}{\bf v}_i$
for ${\bf v}_i=\diag({\bf g}_i)_{i=1}^d{\bf u}$,
${\bf g}_i=G{\bf e}_i$, and $i=1,\dots,d$,
and to $O(n)$
additional flops,
provided the matrix $M\in \mathcal C_{\bf s,t}$
is given with its displacement generator  $(F,G)$
of a length $d$.
(For $d=1$ and $(F,G)=(1,1)$
we arrive at the problems of rational multipoint 
evaluation and interpolation (see part (ii) of Theorem \ref{thevint}).)
Equation (\ref{eqfhr}) reduces  
multiplication
 $C_{\bf s,t}{\bf v}$ 
to multiplication of each of the matrices $V_{\bf s}$
and $V_{\bf t}^{-1}$ by $d+1$  vectors,
to one multiplication of the matrix 
$V_{\bf t}$ by a vector,
and to $O(n)$
additional flops (cf. Theorem \ref{thevint}).
We can apply the new fast algorithms to approximate 
the $2d+3$ 
mat\-rix-by-vec\-tor
products above, but the
approximation errors can readily propagate in this
application of the algorithms.


\section{Conclusions}\label{scnc}


At first we revisited our 
approach of \cite{P90} 
to the transformation of 
matrix structures, covered it 
 comprehensively, 
and simplified its presentation
by employing the Sylvester (rather than Stein) 
displacements and
the techniques for 
operating with them from
\cite{P00} and  \cite[Section 1.5]{P01}. 
Then we singled out a large subclass of 
Cauchy-like matrices,
which we call the
CV-like
 matrices. 
We closely approximated these matrices by
HSS matrices
 and then
applied the Multipole method to the latter HSS matrices.  
This yielded dramatic acceleration of the known
numerical algorithms that approximated
the products of CV and CV-like matrices by vectors 
and the solution of 
 nonsingular linear systems of equations 
with CV and CV-like coefficient matrices.
Namely the running time of the new algorithms is
nearly linear, versus quadratic time required by the known algorithms.
By properly
 transforming matrix structures 
we have readily extended
such an acceleration of the known algorithms
 to
the
matrices
 having  structures of Van\-der\-monde 
and  transposed Van\-der\-monde types,
and consequently 
 to 
numerical multipoint
 evaluation and interpolation
of polynomials.

Potential extensions and specializations include 
computations 
with confluent Vandermonde matrices,
 Loewner matrices,
and  various problems of rational
 interpolation such as the Nevanlinna--Pick and matrix Nehari problems
(cf. \cite[Chapter 3]{P01}
and the bibliography therein),
where, however, the progress 
can be limited to the case of sufficiently well conditioned inputs.
Our 
demonstration of the power
of the transformation of 
matrix structures
should motivate research efforts for
finding 
new inexpensive transforms of matrix structures
and their new algorithmic applications.
Natural topics of further study should
include the  following issues:

(i) extension of our  our approach to a larger 
class of Cauchy and Cauchy-like
 matrices (cf. Remark \ref{reexts}),

(ii) the impact of the conditioning of the input
on the output errors and the running time, 

(iii) the estimation of the treshold 
input sizes for which the proposed algorithms 
running in nearly linear time 
outperform
their variant of Remark \ref{recrd} and 
the known algorithms, running in quadratic time, and

(iv) implementation of the proposed algorithms.

The implementation should be mostly 
reduced to
the application of the Multipole algorithms
and should extend  \cite{XXG12},
but the actual
work 
should prompt the refinements toward
decreasing the treshold values of part (iii).



{\bf Acknowledgements:}
Our research has been supported by NSF Grant CCF--1116736 and
PSC CUNY Awards 64512--0042 and 65792--0043.


\end{document}